\documentclass[preprint,12pt]{elsarticle}

\usepackage{ifpdf}
\usepackage{amsmath}
\usepackage{amsfonts}
\usepackage{amsthm}
\usepackage{mathrsfs}
\usepackage{color}
\usepackage{graphicx}
\usepackage{subfigure}
\usepackage{float}
\usepackage{overpic}

\newcommand{\dss}{\displaystyle}
\newcommand{\tss}{\textstyle}

\biboptions{numbers,sort&compress}

\usepackage{graphicx,natbib,amssymb,lineno}
\ifpdf
\usepackage[%
  pdftitle={Instructions for use of the document class
    elsart},%
  pdfauthor={},%
  pdfsubject={The preprint document class elsart},%
  pdfkeywords={instructions for use, elsart, document class},%
  pdfstartview=FitH,%
  bookmarks=true,%
  bookmarksopen=true,%
  breaklinks=true,%
  colorlinks=true,%
  linkcolor=blue,anchorcolor=blue,%
  citecolor=blue,filecolor=blue,%
  menucolor=blue,pagecolor=blue,%
  urlcolor=blue]{hyperref}
\else
\usepackage[%
  breaklinks=true,%
  colorlinks=true,%
  linkcolor=blue,anchorcolor=blue,%
  citecolor=blue,filecolor=blue,%
  menucolor=blue,pagecolor=blue,%
  urlcolor=blue]{hyperref}
\fi

\makeatletter
\def\elsartstyle{%
    \def\normalsize{\@setfontsize\normalsize\@xiipt{14.5}}
    \def\small{\@setfontsize\small\@xipt{13.6}}
    \let\footnotesize=\small
    \def\large{\@setfontsize\large\@xivpt{18}}
    \def\Large{\@setfontsize\Large\@xviipt{22}}
    \skip\@mpfootins = 18\p@ \@plus 2\p@
    \normalsize
} \@ifundefined{square}{}{} \makeatother

\newtheorem{theorem}{Theorem}[section]
\newtheorem{lemma}[theorem]{Lemma}

\theoremstyle{definition}

\theoremstyle{remark}

\makeatletter
\def\ps@pprintTitle{%
  \let\@oddhead\@empty
  \let\@evenhead\@empty
  \let\@oddfoot\@empty
  \let\@evenfoot\@oddfoot
}
\makeatother

\def\ps@pprintTitle{%
  \let\@oddhead\@empty
  \let\@evenhead\@empty
  \def\@oddfoot{\reset@font\hfil\thepage\hfil}
  \let\@evenfoot\@oddfoot
}

\journal{TBA}
\begin{document}

\begin{frontmatter}

\title{An efficient solution procedure 
for solving \\[0.3ex] 
higher-codimension Hopf and Bogdanov-Takens bifurcations$^\dag$
\footnote{$\!\!^\dag$The first draft of this article has been posted on
arXiv.org since August 28, 2022, No. 2208.13228v1.}
}


\author[mymainaddress]{Bing Zeng}
\ead{zengbing969@163.com}
\author[mysecondaryaddress]{Pei Yu}
\ead{pyu@uwo.ca}
\author[mythirdaryaddress]{Maoan Han\corref{mycorrespondingauthor}}
\cortext[mycorrespondingauthor]{Corresponding author}
\ead{mahan@zjnu.edu.cn}
\address[mymainaddress]{School of Mathematics and Statistics, 
Lingnan Normal University,\\
Zhanjiang, Guandong, 524048, China \vspace*{0.05in}}
\address[mysecondaryaddress]{Department of Mathematics,
Western University, \\ 
London, Ontario, N6A 5B7, Canada \vspace*{0.05in}}
\address[mythirdaryaddress]{Department of Mathematics,    
Zhejiang Normal University,\\
Jinhua, Zhejiang, 321004, China}

\begin{abstract}
In solving real world systems for higher-codimension bifurcation problems, 
one often faces the difficulty in computing the normal form or the focus 
values associated with generalized Hopf bifurcation, and the normal form 
with unfolding for higher-codimension Bogdanov-Takens bifurcation. 
The difficulty is not only coming from the tedious symbolic computation 
of focus values, but also due to the restriction on the system parameters, 
which frequently leads to failure of the conventional approach used in 
the computation even for simple $2$-dimensional nonlinear dynamical systems. 
In this paper, we use a simple $2$-dimensional epidemic model, 
for which the conventional approach fails in analyzing the stability of 
limit cycles arising from  Hopf bifurcation, 
to illustrate how our method can be efficiently applied 
to determine the codimension of Hopf bifurcation. 
Further, we apply the simplest normal form theory to consider  
codimension-$3$ Bogdanov-Takens bifurcation and present 
an efficient one-step transformation approach, compared with the 
classical six-step transformation approach to demonstrate the 
advantage of our method. 
\end{abstract}


\end{frontmatter}

\section{Introduction}

Limit cycle theory plays a very important role in the study of 
nonlinear dynamical systems, related to the well-known phenomenon 
of self-oscillations arising from physical science and 
engineering~\cite{GuckenheimerHolmes1993,HanYu2012}. 
Hopf and Bogdanov-Takens (B-T) bifurcations are two main bifurcations 
generating limit cycles in real world systems. 
A common task of the study in such systems 
is to determine the codimension of the bifurcation
and to derive the associated normal form,
which is not easy for higher-codimension bifurcations. 
Particularly, when considering practical systems, determining 
the codimension of the two bifurcations becomes very difficult
due to physical limitations on the system parameters. 
For example, consider the maximal number of 
limit cycles arising from generalized Hopf bifurcation in a 
2-dimensional nonlinear system, which may be reduced from 
an $n$-dimensional system by applying center manifold theory, 
described by the following ordinary differential equations: 
\begin{equation}\label{Eqn1}
\dot{x} = f(x,\mu,\alpha), \quad 
x \in {\rm R}^2, \ \ \mu \in {\rm R}, \ \ \alpha \in {\rm R}^m, 
\end{equation}
where the dot denotes differentiation with respect to time $t$, 
$\mu$ is a perturbation parameter, and $\alpha$ is a 
constant vector representing the coefficients or parameters in the 
function $f$. Assume that $x=0$ is an equilibrium of \eqref{Eqn1}, yielding 
$f(0,\mu,\alpha)=0$. Moreover, suppose that the Jacobian of the system 
evaluated on the equilibrium $x=0$ at the critical point $\mu=\mu_c=0$
has a pair of purely imaginary eigenvalues $\pm\, i\,\omega_c $. 
Then, applying the normal form theory 
(e.g. see~\cite{Takens1974,Marsden1976,GuckenheimerHolmes1993,
HanYu2012,Kuznetsov1998}) 
to \eqref{Eqn1}, associated with 
the Hopf bifurcation, we obtain the following {\it classical/conventional} 
normal form (CNF) in the polar coordinates, 
\begin{equation}\label{Eqn2}
\begin{array}{rl} 
\dot{r} \!\!\!\! & = r \, (v_0
+ v_1\, r^2 + v_2\, r^4 + \cdots + v_k \, r^{2k} + \cdots ), \\[1.0ex]  
\dot{\theta} \!\!\!\! & = \omega_c 
+ \tau_0 \, \mu + \tau_1\, r^2 + \tau_2\, r^4 + \cdots + \tau_k \, r^{2k} 
+ \cdots, 
\end{array}  
\end{equation}  
where $r$ and $\theta$ are the amplitude and phase of motion, respectively, 
$v_j\ (j=0,1,2, \cdots)$ is called the $j$th-order focus value. 
Note that $v_j$'s are functions of $\alpha$ and $\mu$. 
$v_0$ is obtained 
from a linear analysis as $v_0 \!=\! v_{\rm H}\, \mu$, where 
$v_{\rm H} \ne 0 $ is called the transversal condition of 
Hopf bifurcation, while finding $v_j \ (j\ge 1)$ needs a nonlinear analysis 
such as normal form or focus value computation. 
Note that the frequency $\omega_c$ is usually scaled to $1$.
 
The CNF can be further simplified to the so called 
{\it simplest normal form} (SNF) or {\it unique/minimal normal form} or 
{\it hypernormal form} (e.g, see~\cite{Baider1991,
Algaba1998,Ushiki1984,Yu1999,GY2010,GY2012,YL2003,GM2015,YuZhang2019}). 
The SNF of Hopf bifurcation can be classified into 
three categories (at the Hopf critical point with $\mu=0$, i.e.,
$v_0 \!=\! 0$)~\cite{Yu1999} as follows: 
\begin{equation}\label{Eqn3}
\begin{array}{cl}
\textrm{(I)} 
& v_1 \ne 0 \!: \ \ \left\{
\begin{array}{ll}
\dot{r} = v_1 \rho^3 + v_2 \rho^5, \\[0.5ex]
\dot{\theta} = 1 + \tau_1 \rho^2; 
\end{array}
\right. 
\\[4.0ex] 
\textrm{(II)} 
& \left\{\!\!
\begin{array}{ll}
v_1 = v_2 = \cdots = v_{k-1} = 0, \, v_k \ne 0 \\
\tau_1 = \tau_2 = \cdots = \tau_{k-1} = 0
\end{array}
\right.\! :
\quad \left\{\!\!
\begin{array}{ll}
\dot{r} = v_k \rho^{2k+1} + v_{2k} \rho^{4k+1}, \\[0.5ex]
\dot{\theta} = 1 + \tau_k \rho^{2k}; 
\end{array}
\right. 
\\[4.0ex] 
\textrm{(III)} 
& \left\{\!\!
\begin{array}{ll}
v_1 = v_2 = \cdots = v_{k-1} = 0,\, v_k \ne 0 \\
\tau_1 = \tau_2 = \cdots = \tau_{j-2} = 0,\, \tau_{j-1} \ne 0, \
(1 \le j \le k) 
\end{array}
\right.\! :
\\[3.0ex]
& \left\{
\begin{array}{ll}
\dot{r} = v_k \rho^{2k+1} + v_{2k} \rho^{4k+1}, \\[0.5ex]
\dot{\theta} = 1 + \tau_{j-1} \rho^{2(j-1)}
+ \tau_j \rho^{2j} + \cdots + \tau_k \rho^{2k}. 
\end{array}
\right.  
\end{array}  
\end{equation}  
It can be seen that the CNF \eqref{Eqn2} 
contains an infinite number of ``tails'', while the SNF 
\eqref{Eqn3} has only a finite number of terms, since the infinite tails 
in the CNF have been removed by a further arbitrarily high-order 
nonlinear transformation. 
Then, the codimension of Hopf bifurcation is defined by the first 
non-vanishing focus value. Thus, the codimension of Hopf bifurcation 
is $1$ for the case (I), and $k$ for the cases (II) and (III). 
However, it should be noted that the above conclusion is based on the 
assumption that the vector parameter $\alpha$ is real (without any 
additional restriction), and therefore the number $k$ can usually reach 
its maximal value. 

In solving Hopf bifurcation problems, the standard approach is to 
compute the focus values (or the normal form) of the system 
associated with a Hopf bifurcation from an equilibrium solution. 
The computation is often carried out with the aid of a computer 
algebraic software such as Maple or Mathematica. 
Then, one needs to solve a multi-variate polynomial system 
based on the normal form 
or the focus values. There are two main difficulties in dealing with the 
problems related to the above normal forms. The first one is due to 
the symbolic computational complexity in the focus value (or the normal form) 
computation, which is a result of the application of the conventional 
approach used in stability and bifurcation analysis. 
This will be seen in the next section when we 
deal with Hopf bifurcation in a simple epidemic model. 
The second difficulty is owing to that practical systems often have extra 
restriction on the system parameters because system parameters 
must be positive or even restricted to certain limited values.  
Suppose that the system under consideration involves $4$ real parameters.
In general, if these parameters are assumed real, 
then the maximal number of bifurcating limit cycles may be $4$, the same as 
the number of parameters. However, if it is a biological system or other 
physical systems, due to limitation on the parameters, the maximal number 
of limit cycles might be $3$, $2$, or even only $1$. In this case, 
determining the codimension of the Hopf bifurcation, that is, determining 
the maximal number of bifurcating limit cycles can be much more difficult. 
The difficulty is mainly from solving the polynomial systems (suppose 
the focus values have been obtained), since one needs to determine the 
sign of the polynomials with the variation of many variables (parameters).   

For the Bogdanov-Takens (B-T) bifurcation, the analysis of codimension-$2$ 
B-T bifurcation has become 
standard~\cite{GuckenheimerHolmes1993,Kuznetsov1998}. 
However, for codimension-$3$ or higher-codimendion (or degenerate) 
B-T bifurcations, the computation of the normal forms becomes much more 
involved, particularly in order to establish the relation between 
the original system and the simplified system (the normal form). 
Consider the system \eqref{Eqn1} which now has a nilpotent critical 
point at the origin (characterized by a double-zero eigenvalue), with more
than one perturbation parameters (unfolding), which is rewritten as  
\begin{equation}\label{Eqn4}
\dot{x} = f(x,\mu,\alpha), \quad 
x \in {\rm R}^2, \ \ \mu \in {\rm R}^p, \ (p \geqslant 2), 
\ \ \alpha \in {\rm R}^m, 
\end{equation}
Then, the CNF of B-T bifurcation for system \eqref{Eqn4} at the 
critical point $\mu \!=\! 0$ can be written as (e.g., see 
\cite{Takens1974,GuckenheimerHolmes1993,HanYu2012,Kuznetsov1998})
\begin{equation}\label{Eqn5} 
\begin{array}{ll} 
\dot{x}_1 = x_2, \\[1.0ex]  
\dot{x}_2 = \dss\sum_{k=2}^\infty \left( c_{k0}\, x_1^k + c_{(k-1) 1}\, 
x_1^{k-1} x_2 \right),  
\end{array}  
\end{equation} 
where the normal form coefficients $c_{k0}$ and $c_{(k-1) 1}$ 
are functions of the vector parameter $\alpha$ and $\mu$. 
Unlike the SNF of Hopf bifurcation, the classification of the SNF 
of \eqref{Eqn4} is much more complicated. The most common case in 
real applications is the cusp B-T bifurcation and degenerate 
cusp B-T bifurcations when $c_{20} \ne 0$. 
The SNF for such B-T bifurcations is    
given by \cite{Han1997,YuZhang2019}
\begin{equation}\label{Eqn6}
\begin{array}{ll} 
\dot{x}_1 = x_2, \\[1.0ex]  
\dot{x}_2 = c_{20}\, x_1^2 + c_{11}\, x_1 x_2    
+ c_{31}\, x_1^3 x_2 + c_{41}\, x_1^4 x_2 + c_{61}\, x_1^6 x_2 + 
c_{71}\, x_1^7 x_2 + \cdots
\end{array} 
\end{equation} 
which can be used to estimate the codimension of the B-T bifurcation. 
Besides the case $c_{20} \! \ne \! 0$, the first non-vanishing coefficients 
determines the codimension. For example, 
it is a codimension-$2$ (cusp) B-T bifurcation when $c_{11} \!\ne\! 0$; 
a codimension-$3$ (degenerate cusp) B-T bifurcation when $c_{11} \!=\! 0$ 
and $c_{31} \! \ne \! 0$; 
and a codimension-$4$ (degenerate cusp) B-T bifurcation when 
$c_{11} \!=\! c_{31} \!=\! 0$ and $c_{41} \! \ne \! 0$; and so on. 
In general, the first non-vanishing coefficient can be written as 
$c_{j1}$, where $j \!=\! \big[\frac{3(k-1)}{2}\big],\ (k \geqslant 2)$. 
However, 
whether or not such an estimate of the codimension is true depends upon  
the derivation of the unfolding expressed in the perturbation parameters. 
This leads to much more computational demanding and will be seen in 
section 3. 

The codimension-$3$ degenerate cusp B-T bifurcation 
(when $c_{20} c_{31} \ne 0, \, c_{11} = 0$)
was studied by Dumortier {\it et al.} in 1987~\cite{Dumortier1987}, 
and the classical six-step transformation approach was developed 
and widely used by researchers in deriving the SNF with unfolding. 
We remark that some mistakes appeared in discussing B-T bifurcation 
in~\cite{Dumortier1987}, and were pointed out and corrected 
in~\cite{HLY2018}. 
Recently, the so-called one-step transformation method was proposed 
\cite{YuZhang2019,ZengDengYu2020}, which provides the transformation for 
the state variables, the parameters and the time rescaling in just one step, 
yielding a direct relation between the original system and the SNF. 
This not only greatly simplifies the analysis, but also clearly shows 
the impact of the original system parameters on the dynamical behaviours 
of the system. This method is based on the SNF 
theory and the parametric simplest normal (PSNF) theory 
\cite{Yu1999,GY2010,GY2012,YL2003,GM2015,YuZhang2019}. 
The key step involved in this method 
is to choose appropriate bases for the SNF and 
PSNF, as well as in the nonlinear transformations. 

In nowadays, using computer software package such as MATCONT~\cite{DGK2003}
or XPPAUTO \cite{Roussel2019} to plot bifurcation diagrams of nonlinear 
dynamical systems becomes very popular and quite useful in applications, 
particularly for lower-codimension bifurcations such as saddle-node, 
transcritical, Hopf and B-T bifurcations. The basic idea is 
to use computer simulation to search the critical bifurcation points/curves 
in the parameter space, and the bifurcation diagram is usually plotted in 
a $2$-dimensional parameter plane. However, such techniques do not 
provide analytical formulas in terms of parameters for a parametric 
study in designs. Also, they are not applicable for higher-codimension 
bifurcations. Therefore, it is necessary to develop efficient methods 
for the analysis of higher-codimension bifurcations.   
 
In this paper, we will use a simple epidemic model to illustrate 
how to determine the codimension of Hopf and B-T bifurcations. 
In particular, we will show how to determine the codimension of 
Hopf bifurcation, and introduce both the six-step and one-step 
transformation approaches for the codimension-$3$ (degenerate cusp) 
B-T bifurcation to give a comparison. The epidemic model has been studied in 
\cite{LiZhouWuMa2007} for Hopf bifurcation and codimension-$2$ 
B-T bifurcation. Later, Li {\it et al.}~\cite{LiLiMa2015}
gave a complete analysis on the codimension-$3$ B-T bifurcation 
using the six-step method. The simple SI-epidemic model is described 
by the following differential equations, 
\begin{equation}\label{Eqn7}
\begin{array}{rl}
\dfrac{{\rm d} S}{{\rm d}t} = \!\!\! & A - dS - \beta (1+\varepsilon I) SI,
\\[2.0ex]
\dfrac{{\rm d} I}{{\rm d}t} = \!\!\! & \beta (1+\varepsilon I) SI 
- (d + \alpha) I, 
\end{array}
\end{equation}
where $S$ and $I$ represent the numbers of the susceptible 
and infective populations, respectively; $A$, $d$ and $\alpha$ denote 
the recruitment rate of susceptibles, the nature death rate, and 
the sum of the recover rate and the disease-related death rate, 
respectively; and $\beta (1 + \varepsilon I)SI $ is the incidence rate. 
All the parameters $A,\, d,\, \alpha,\, \beta $ and 
$\varepsilon$ take positive real values. 

In \cite{LiZhouWuMa2007}, the authors use  
$I=X$, $N=S+I=Y$ and apply the rescaling $ \tau = \alpha\, t$ 
to model \eqref{Eqn7} to obtain the following dimensionless system, 
\begin{equation}\label{Eqn8}
\begin{array}{rl}
\dfrac{{\rm d} X}{{\rm d} \tau} 
= \!\!\! & X \big[ k (1+\varepsilon X)(Y-X) -(n + 1) \big], \\[2.0ex]
\dfrac{{\rm d}Y}{{\rm d} \tau} = \!\!\! & m - n Y  - X,
\end{array}
\end{equation} 
where the new parameters are defined as 
\begin{equation}\label{Eqn9}
k = \dfrac{\beta}{\alpha}, \quad 
m= \dfrac{A}{\alpha}, \quad n = \frac{d}{\alpha}. 
\end{equation} 
Later, this model was further studied by Zeng and Yu 
\cite{ZengYu2021} for a detailed analysis on Hopf bifurcation.  
Note that the dimensionless model \eqref{Eqn8} contains $4$ parameters.
However, one can make a further transformation, as given by 
\begin{equation}\label{Eqn10}
X = m\,x, \quad Y = m\,y, \quad  k = \dfrac{1}{m}\, \kappa, \quad 
\varepsilon = \dfrac{1}{m}\, e,  
\end{equation}
to eliminate the parameter $m$, yielding 
\begin{equation}\label{Eqn11}
\begin{array}{rl}
\dfrac{{\rm d} x}{{\rm d} \tau} 
= \!\!\! & x\, \big[ \kappa (1+ e\, x)(y-x) -(n + 1) \big], \\[2.0ex]
\dfrac{{\rm d}y}{{\rm d} \tau} = \!\!\! & 1 - n y - x.
\end{array}
\end{equation} 
In other words, under the transformation \eqref{Eqn10}, 
without loss of generality, one can simply set $m=1$ in system \eqref{Eqn8}. 
In this paper, we will use system \eqref{Eqn11} for bifurcation analysis. 
In fact, using the system \eqref{Eqn8} shows that 
$m$ does not play any roles on the bifurcation analysis. 
However, in order to keep our simulation results consistent with those 
given in \cite{LiZhouWuMa2007} and \cite{ZengYu2021}, 
our simulations presented in this paper are still based on 
system \eqref{Eqn8}, and all the notations introduced later for 
system \eqref{Eqn11} can be easily extended to system \eqref{Eqn8}
with the transformation \eqref{Eqn10}.  

It should be pointed out that most epidemic models
have the well-posedness property, that is, solutions of such a model 
remain positive if the initial points take positive values, and are bounded. 
However, for the system \eqref{Eqn11} (or the system \eqref{Eqn8}), 
the first quadrant in the $x$-$y$ plane is not invariant.
Trajectories starting from the initial points in the first quadrant may
pass through the $x$-axis to enter the fourth quadrant and then return
to the first quadrant. Since the $y$-axis is invariant, 
any trajectories starting from the initial points in the first or 
fourth quadrant will either remain in or eventually enter the first quadrant.
In other words, if restricted to the region:
$\{(x,y) \! \mid \! x \! \geqslant \! 0 \}$, the well-posedness property on
the solutions of \eqref{Eqn8} is well defined. 
Although this model is not perfect, we do not intend to improve 
it since the aim of this paper is to use this model to 
demonstrate a solution procedure for studying higher-codimension Hopf and B-T 
bifurcations. As discussed in \cite{LiZhouWuMa2007}, we may focus on 
the domain of interest for the model, defined by 
\begin{equation}\label{Eqn12} 
\begin{array}{rll} 
& \Omega = \left\{(x,y) \left|\, 0 \leqslant 
x< y \leqslant \dfrac{1}{n} \right. \right\} & \textrm{for system } 
\eqref{Eqn11}, \\[2.5ex] 
\textrm{or} \ \ & 
\Omega = \left\{(X,Y) \left|\, 0 \leqslant 
X< Y \leqslant \dfrac{m}{n} \right. \right\} & \textrm{for system }  
\eqref{Eqn8}. \\[1.0ex] 
\end{array}  
\end{equation} 
Note that $\Omega$ does not serve as a trapping region 
for the dynamical solutions in the first quadrant.  

In the next section, we study Hopf and generalized Hopf bifurcations
of system \eqref{Eqn11} and focus on the study of codimension. 
Then, in Section 3 we consider the B-T bifurcation in system \eqref{Eqn11} 
and pay particular attention to codimension-$3$ B-T bifurcation. 
Various simulations showing different bifurcation phenomena are given 
to illustrate the theoretical predictions. 
A concluding remark is given in Section 4.

\section{Hopf bifurcation of system \eqref{Eqn11}}

In this section, we first derive the equilibrium solutions 
of system \eqref{Eqn11} and their stability, and then consider 
the maximal number of limit cycles 
which may bifurcate from Hopf critical points. 
Although the stability conditions for the equilibria of system \eqref{Eqn8} 
were given in \cite{LiZhouWuMa2007}, the analysis on 
Hopf bifurcation has not been completely explored. 
In particular, we will rigorously prove that the codimension of 
the Hopf bifurcation is two.  
When using a classical method, one usually expresses equilibrium 
solutions in terms of the system parameters. The advantage of this 
approach is to show the dynamical behaviours of the system, 
such as stability and bifurcations, clearly in the parameter space. 
However, if the equilibrium solutions cannot be simply expressed 
in terms of the system parameters, for example, if they are determined 
by a quadratic equation, then the analysis on stability and bifurcations 
becomes much more involved. Especially, it causes more difficulty in 
computing normal forms (focus values) and it is almost impossible to 
determine whether a focus value can change its sign or not, which is 
directly related to determining the codimension of Hopf bifurcation.

\subsection{Stability of bifurcating limit cycles} 

First, we derive the conditions for the existence of 
the equilibrium solutions of system \eqref{Eqn11} 
and their stability. We will give a complete partition in the 
parameter space for the bifurcation analysis.  
Setting $ \frac{{\rm d} x}{{\rm d} \tau} \!=\! \frac{{\rm d} y}{{\rm d} \tau} 
\!=\! 0$ in system \eqref{Eqn11} yields two equilibrium solutions, 
\begin{equation}\label{Eqn13} 
\begin{array}{ll} 
{\rm P_0}: & (x_0,\, y_0) = \Big(0, \, \dfrac{1}{n} \Big), \\[1.0ex] 
{\rm P_1}: & (x_1,\, y_1) = \big(1-n y_1, \, y_1 \big), \quad 
\Big( 0 < y_1 < \dfrac{1}{n} \Big),  
\end{array} 
\end{equation}
where ${\rm P_0}$ is the infection-free equilibrium (boundary equilibrium) 
and ${\rm P_1}$ is the infectious equilibrium (positive equilibrium), with 
$y_1$ determined from the following quadratic polynomial, 
\begin{equation}\label{Eqn14} 
g(y_1,\kappa) = 1 + \kappa\, e\, n (y_1 - y_{1*}) (y_1 - y_1^*), \quad  
\textrm{where} \ \ y_{1*} = \dfrac{1}{n + 1}, \ 
y_1^* = \dfrac{1}{n} + \dfrac{1}{e\,n}.  
\end{equation} 
Solving $g=0$ gives the infectious equilibrium solutions, 
\begin{equation}\label{Eqn15}
y_{1\pm} = \dfrac{1}{2 \kappa \,e\, n(n+1)} 
\left\{ \kappa \big[e\, (2n+1) + n+1 \big] \pm \sqrt{\Delta} \right\},  
\quad x_{1\pm} = 1 - n y_{1\pm}, 
\end{equation} 
where 
\begin{equation}\label{Eqn16} 
\Delta = \kappa\, \big[ \kappa \,(e +n +1)^2 - 4\,e\, n (n+1)^2 \big]. 
\end{equation}
For convenience, define 
\begin{equation}\label{Eqn17} 
\begin{array}{rl} 
{\rm P_{\rm 1 \pm}} = \!\!\! & ( 1 - n\, y_{1 \pm}, \, y_{1 \pm} ), \\[1.0ex] 
e_1 = \!\!\! & n + 1, \quad 
e_2 = 4 n\, e_1^2, \quad 
e_3 = \dfrac{e_1^2}{1 - n}, \ (n < 1), \quad 
e_4 = \dfrac{e_1^2}{n},
\\[2.5ex]
e_\pm = \!\!\! & \dfrac{e_1}{2 n}
\left[ n+1 \pm \sqrt{(n+1) (1-3 n)} \right], \quad \Big( n \leqslant
\dfrac{1}{3} \Big),
\\[2.5ex]
\kappa_{\rm T} = \!\!\! & n\, e_1, \quad 
\kappa_{\rm SN}=\dfrac{e\, e_2} {(e + e_1)^2}, \quad 
\kappa^* = \dfrac{2 n \,e\, e_1 e_4} {(e + e_1) (e + e_4 )}, 
\\[2.5ex] 
\kappa_{\rm H_\pm} 
= \!\!\! & \dfrac{ e\, [\,n\,(e+e_4)+e] \pm n\, (e_4 - e) 
\sqrt{e (e -e_2)}} {2 (e + e_1)^2}, 
\\[2.5ex] 
y_{\rm 1SN} = \!\!\! & \dfrac{1}{e_1} 
+ \dfrac{1}{2n} \Big(\dfrac{1}{e} + \dfrac{1}{e_1} \Big)
\in (y_{1*},\, y_1^*), \quad y_{\rm 1T} = \dfrac{1}{n}
\in (y_{1*},\, y_1^* ), 
\\[2.5ex] 
{\rm R_0} = \!\!\! & \dfrac{\kappa}{n(n+1)} \stackrel{\triangle}{=} 
\dfrac{\kappa}{\kappa_{\rm T}}, 
\end{array}
\end{equation}
where ${\rm R_0}$ is the basic reproduction number, 
and the subscripts T, SN and H represent Transcritical, Saddle-node 
and Hopf bifurcations.  
All the above notations for $y_1$, $e$ and $\kappa$ can be similarly 
defined for $Y_1$, $\varepsilon$ and $k$ of system \eqref{Eqn8} via the 
transformation \eqref{Eqn10} as follows: 
\begin{equation}\label{Eqn18} 
\begin{array}{rllll} 
Y_{1\pm} = \!\!\! & m\, y_{1\pm}, & 
Y_{\rm 1SN}=m\, y_{\rm 1 SN}, & Y_{\rm 1T}=m\, y_{\rm 1T}, & 
Y_{1*}=m\, y_{1*}, \ Y_1^* = m\, y_1^* \\[1.0ex] 
\varepsilon_i = \!\!\! & \dfrac{e_i}{m}, & \hspace{-0.25in} i=1,2,3,4, & 
\varepsilon_{\pm}  = \dfrac{e_\pm}{m}, \\[2.0ex]
k_{\rm T}= \!\!\! & \dfrac{\kappa_{\rm T}}{m}, &  
k_{\rm SN}= \dfrac{\kappa_{\rm SN}}{m}, &  
k^* = \dfrac{\kappa^*}{m}, &  
k_{\rm H_\pm} = \dfrac{\kappa_{\rm H_\pm}}{m}, 
\end{array}
\end{equation}
and $ X_{1\pm} \!=\! m \!-\! n Y_{1\pm}$. The notations given in 
\eqref{Eqn18} will be used in figures for simulations and 
in bifurcation diagrams.  

In the following analysis, ${\rm P_1}$ denotes ${\rm P_{1\pm}}$. 
It is easy to show that $\kappa_{\rm T} \!>\! \kappa_{\rm SN}$. 
Further, we treat $\kappa$ as a bifurcation 
parameter, and the other two parameters $e$ and $n$ 
as control parameters. In addition, we call the Type-I bistable phenomenon 
(or Type-I coexistence of bistable states) 
if two stable equilibria coexist, and the Type-II bistable phenomenon 
(or Type-II coexistence of bistable states)
if a stable equilibrium and a stable limit cycle coexist.   

In \cite{ZengYu2021}, the following lemma about the stability and 
bifurcation of the equilibria ${\rm P_0}$ and ${\rm P_1}$ has been 
proved. However, whether the Hopf bifurcations are supercritical 
or subcritical is not proved in \cite{ZengYu2021} since the conventional 
analytical method does not work due to the complex expressions of 
$y_{1-}$ and $\kappa_{\rm H_{\pm}}$, which makes it impossible to compute 
the focus values or the normal forms associated with Hopf bifurcations.  
In the following, we first derive the explicit 
conditions on the parameters to classify the types of Hopf bifurcations, 
and then consider the codimension of Hopf bifurcations. 
For the readability of the readers, we list Theorem 2.1 in \cite{ZengYu2021} 
for system \eqref{Eqn8} as a lemma below, with a modification to 
adapt system \eqref{Eqn11}. 

\begin{lemma}\label{Lem1}
For the system \eqref{Eqn11},
the infection-free equilibrium  ${\rm P_0}$ is
asymptotically stable if the basic reduction number, 
${\rm R_0} \!<\! 1\ ($i.e., $ \kappa \!<\! \kappa_{\rm T})$, 
and unstable if ${\rm R_0} \!>\! 1\ ($i.e., $ \kappa \!>\! \kappa_{\rm T})$.
The infectious equilibrium
${\rm P_1}$ does not exist for $\kappa \!<\! \kappa_{\rm SN}$;
${\rm P_{1-}}$ exists only for $\kappa \!\geqslant\! \kappa_{\rm T}$ and 
$ e \! \leqslant \! e_1$; both ${\rm P_{1\pm}}$ exist
for $\kappa \!>\! \kappa_{\rm SN}$ and $ e \! > \! e_1$,
with ${\rm P_{1+}}$ beign a saddle point.
A transcritical bifurcation occurs between ${\rm P_0}$ and ${\rm P_1}$
at the critical point $\kappa \!=\! \kappa_{\rm T}$.
Hopf bifurcations can occur from the equilibrium ${\rm P_{1-}}$ 
under certain conditions on the parameters. More details are
given below.
\begin{itemize}
\item[{\rm (I)}]
When $e \! \leqslant \! e_1 $, no bistable phenomenon can happen. Moreover, 

\begin{itemize}
\item[{\rm (I-a)}]
if $e \!\leqslant\!  \min\{ e_1,e_2\}$, then 
${\rm P_{1-}}$ is asymptotically
stable for $\kappa \! > \! \kappa_{\rm T}$, and no Hopf bifurcation can happen;

\item[{\rm (I-b)}]
if $n \! \leqslant \! \frac{\sqrt{2}-1}{2}$ and
$ e_2 \!<\! e \! \leqslant \! e_1$,
then two Hopf bifurcations occur at $\kappa \! =\! \kappa_{\rm H_-}$ and 
$\kappa \!=\! \kappa_{\rm H_+}$; ${\rm P_{1-}}$ is asymptotically stable for
$\kappa \! \in \! (\kappa_{\rm T},\, \kappa_{\rm H_-}) \bigcup 
(\kappa_{\rm H_+}, \infty)$,
and unstable for $\kappa \! \in \! (\kappa_{\rm H_-},\, \kappa_{\rm H_+})$.
\end{itemize}

\item[{\rm (II)}]
When $e \!>\! n+1$, the following holds. 

\begin{itemize}
\item[{\rm (II-a)}]
${\rm P_{\rm 1-}}$ is asymptotically stable, and no Hopf bifurcation 
can happen if one of the following
conditions is satisfied:

{\rm (II-a-i)} $ e \!\geqslant\! e_4$ and
$ \kappa \!>\! \max\{\kappa_{\rm SN},\, \kappa^*\}$;

{\rm (II-a-ii)} $n \!<\! 1$, $e \!>\! \max\{e_3, e_4\}$ and $\kappa=\kappa^*$; 

{\rm (II-a-iii)} $n \! \geqslant \! \frac{\sqrt{2}-1}{2}$,
$e_1 \!<\! e \!<\! \min\{e_2, e_4 \}$ and $ \kappa \!>\! \kappa_{\rm SN} \ 
(>\kappa^*)$.

Type-I bistable phenomenon can occur in all the above three subcases.

\item[{\rm (II-b)}]
If $ 0 \!<\! n \!<\! \frac{1}{2}$,
$ e_3 \!<\! e \! \leqslant \! e_4$
and $ \kappa_{\rm SN} \!<\! \kappa \! \leqslant \! \kappa^*$, then 
${\rm P_{1-}}$ is unstable, excluding Hopf bifurcation.  

\item[{\rm (II-c)}] 
One Hopf bifurcation occurs in the following cases. 
\begin{itemize}
\item[{\rm (II-c-i)}]
If $n \!<\! \frac{1}{2}$, $ e_3 \!<\! e \!<\! e_4$
and $\kappa \!>\! \kappa^* \ (> \! \kappa_{\rm SN})$,
then a Hopf bifurcation occurs at $ \kappa\!=\! \kappa_{\rm H_+}$;
${\rm P_{\rm 1-}}$ is asymptotically stable for
$\kappa \!\in\! ( \kappa_{\rm H_+}, \infty)$,
and unstable for $\kappa \! \in \! (\kappa_{\rm SN},\kappa_{\rm H_+})$.
Both type-I bistable states and Type-II bistable states coexist
if $\frac{1}{3} \! \leqslant \! n \!<\! \frac{1}{2}$,
or if $n \!<\! \frac{1}{3}$ with $e \!> e_+$,
for which $\kappa_{\rm H_+} \!<\kappa_{\rm T}$.

\item[{\rm (II-c-ii)}]
If $n \!<\! 1$, $e \!>\!  \max\{ e_3, e_4 \}$
and $\kappa_{\rm SN} \!<\! \kappa \!<\! \kappa^*$,
then a Hopf bifurcation occurs at $ \kappa\!=\! \kappa_{\rm H_+}$;
${\rm P_{\rm 1-}}$ is asymptotically stable for $\kappa \!>\! 
\kappa_{\rm H_+}$, and unstable for $\kappa \! \in \! 
(\kappa_{\rm SN},\kappa_{\rm H_+})$.
Both Type-I and Type-II bistable states coexist,
since $\kappa_{\rm H_+} \!<\kappa_{\rm T}$.
\end{itemize}

\item[{\rm (II-d)}]
Two Hopf bifurcations occur at $ \kappa \!=\! \kappa_{\rm H_-}$
and $ \kappa \!=\! \kappa_{\rm H_+}$ if $ n \!<\! \frac{1}{2}$,
$ \max\{e_1, e_2\} $ $ <\! e \!<\! e_3$
and $\kappa \!\geqslant \! \kappa_{\rm SN} \ (>\! \kappa^*)$;
${\rm P_{\rm 1-}}$ is asymptotically stable for
$\kappa \!\in \! (\kappa_{\rm SN},\kappa_{\rm H_-}) \bigcup 
(\kappa_{\rm H_+},\infty)$,
and unstable for $ \kappa \! \in \! (\kappa_{\rm H_-},\kappa_{\rm H_+})$.
Type-I bistable states coexist, and Type-II bistable states coexist
if $ \frac{\sqrt{5}-1}{4} \! \leqslant \! n \!<\! \frac{1}{2}$
and $ e_2 \!<\! e \!<\! e_3 $, or $n \!<\! \frac{\sqrt{5}-1}{4}$ and
$ e_- \!<\! e \!<\! e_3$ for
which $\kappa_{\rm H_-} \! <\! \kappa_{\rm T}$; and
if $ \frac{\sqrt{5}-1}{4} \! \leqslant \! n \!<\! \frac{1}{3}$ and
$ e_2 \!<\! e \!<\! e_-$, 
or $ \frac{1}{3} \!\leqslant\! n \!<\! \frac{1}{2}$ and
$e_2 \!<\! e \!<\! e_3$  for which $\kappa_{\rm H_+} \! <\! \kappa_{\rm T}$.
\end{itemize}
\end{itemize} 
\end{lemma}

\begin{figure}[!h]
\vspace*{-0.20in}
\begin{center} 
\vspace{-2.00in}
\hspace*{-0.50in}
\begin{overpic}[width=1.2\textwidth,height=1.12\textheight]{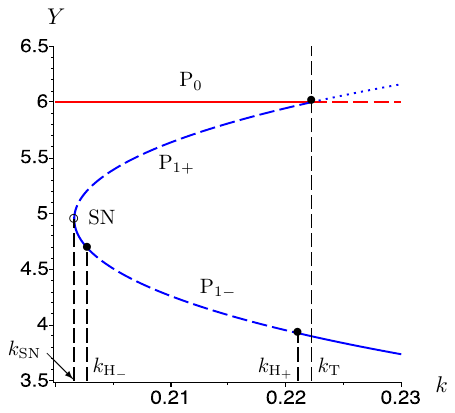}
\end{overpic}

\vspace{-4.90in}
\caption{Bifurcation diagram for the model \eqref{Eqn8}
projected on the $k$-$Y$ plane 
with $m=2$, $n=\frac{1}{3}$ and $ \varepsilon = \frac{5}{4}$,
corresponding to \vspace{0.03in} the Case (II-d) in Lemma~\ref{Lem1} 
with $k^* \!=\! \frac{640}{3243}$,
$k_{\rm SN} \!=\! \frac{320}{1587}$ 
and $k_{\rm SN} \!=\! \frac{320}{1587}$, 
having two Hopf critical points at 
$k_{\rm H_\pm} \!=\! \frac{4035 \pm 17 \sqrt{105}}{19044}$. 
The colored solid and dashed curves denote
stable and unstable equilibria, respectively, while the dotted curve 
represents a mathematical solution without biological meaning.}
\label{fig1} 
\end{center}
\end{figure}

The bifurcation diagram for the Case (II-d) is given in Figure~\ref{fig1}
for system \eqref{Eqn8},
with the parameter values: $m\!=\!2$, $n \!=\! \frac{1}{3}$ and 
$\varepsilon \!=\! \frac{5}{4}$. 
Note that the notations given in \eqref{Eqn18} are used in 
the bifurcation diagrams. 
First note that ${\rm P_1}$, determined from $g \!=\! 0$,
only exists for $ k \! \geqslant \! k_{\rm SN}$. 
However, the part of the solution $Y_1$ satisfying $Y_1 \!>\! Y_{\rm 1T} 
$ is biologically meaningless since $X_1 \! < \! 0$ when 
$Y_1 \!>\! Y_{\rm 1T}$. 
On the bifurcation diagram, projected on the $k$-$Y$ plane,
$Y_1 \!=\! Y_{1*}$ and $Y_1 \!=\! Y_1^*$ are two horizontal 
asymptotes of the curve $g \!=\! 0$ (which are not shown in 
the diagram), serving as the lower and upper 
boundaries of the solution ${\rm P_1}$. 
The curve has a unique vertex at $(k_{\rm SN}, Y_{\rm 1 SN})$.  
Moreover, using the derivative $\frac{{\rm d}k}{{\rm d} Y_1}$, 
we can show that the solution $Y_1$, determined by a function 
$k \!=\! k(Y_1)$, is monotonically decreasing for $Y_1 \!<\! Y_{\rm 1SN}$ 
and monotonically increasing for $Y_1 \!>\! Y_{\rm 1SN}$, like a parabola. 
Hence, when $Y_{\rm 1SN} \!\geqslant \! Y_{\rm 1T}$, i.e., 
when $\varepsilon \! \leqslant \! \varepsilon_1$, ${\rm P_1}$ has one solution
${\rm P_{1-}}$; while when $Y_{\rm 1SN} \! < \! Y_{\rm 1T}$, i.e., 
when $\varepsilon \! > \! \varepsilon_1$, ${\rm P_1}$ has two solutions: 
${\rm P_{1+}}$ and ${\rm P_{1-}}$, and ${\rm P_{1+}}$ exists for 
$ Y_{\rm 1SN} \! \leqslant \! Y_1 \! \leqslant Y_{\rm 1T}$, 
while ${\rm P_{1-}}$ exists for $ Y_{1*} \!<\! Y_1 \! \leqslant \! 
Y_{\rm 1SN}$. 
It is shown in the figure that there exist two Hopf bifurcations, and 
both Type-I and Type-II
bistable phenomena exist because the chosen parameter values satisfy 
$k_{\rm SN} \!<\! k_{\rm H_-} \! < k_{\rm H_+} \!<\! k_{\rm T}$.
Hence, two stable equilibria ${\rm P_0}$ and ${\rm P_{1-}}$ 
coexist for $k \! \in \! (k_{\rm SN},k_{\rm H_-}) \bigcup    
(k_{\rm H_+},k_{\rm T})$, 
while stable ${\rm P_0}$ and a stable limit 
cycle (which is verified by simulation and needs a rigorous proof) 
coexist for $ k \! \in \! (k_{\rm H_-},k_{\rm H_+})$.

Whether the Hopf bifurcations for the cases
(I-b), (II-c) and (II-d) in Lemma~\ref{Lem1} 
are supercritical or subcritical depends on the sign of the first-order 
focus value associated with the Hopf bifurcation. 
However, if one uses the expressions of the solution $y_{1-}$ 
and the critical points $\kappa_{\rm H_{\pm}}$ given in \eqref{Eqn17} to 
derive the first-order focus value, it is impossible to compute 
the first-order focus value since the resulting equations are 
too complex to deal with. 
We will use a parameter which is linear in the function $g$, 
instead of $y_1$, to solve $g=0$. $y_1$ is then treated as 
a ``parameter'' in the stability and bifurcation analysis, because it 
is a component of the equilibrium solution ${\rm P_1}$ and is indeed 
a function of the system parameters. Consequently, the stability conditions 
on the parameters need to be rederived using $y_1$ and the other parameters. 

We have following result for determining whether the Hopf bifurcations 
in Lemma~\ref{Lem1} are supercritical or subcritical.  

\begin{theorem}\label{Thm2.2} 
Hopf bifurcation of system \eqref{Eqn11} exists for 
$\,0 < n <\! 1$, and it is supercritical for 
$0 \!<\! n \! \leqslant \! \frac{1}{3}$, and subcritical for 
$ \frac{1}{2} \! \leqslant \! n \!<\! 1$. When $\frac{1}{3} \!<\! 
n \!<\! \frac{1}{2}$, Hopf bifurcation is supercritical if 
$ e \!<\! e^*$, and subcritical if 
$e \!>\! e^*$, where  
\begin{equation}\label{Eqn19} 
\begin{array}{rl} 
e^* \!\!\!\! & 
= \dfrac{n (n+1)^2}{(1-2n) \sqrt{(1+n) (1-2 n)} + (n^2+2 n-1)},
  \quad n \in \Big( \dfrac{1}{3}, \dfrac{1}{2} \Big).  
\end{array} 
\end{equation} 
For the cases in Lemma~\ref{Lem1}, 
the Hopf bifurcations in the Cases {\rm (I-b), (II-c-i)} and 
{\rm (II-d)} are supercritical. 
For the Case {\rm (II-c-ii)}, the Hopf bifurcation is 
supercritical when $ 0 \!<\! n \! \leqslant \! \frac{1}{3} $, or
when $ \frac{1}{3} \!< \! n \!<\! \frac{1}{2} $ and  
$\max\{e_3,e_4\} \!<\! e \!<\! e^*$; 
and subcritical when $\frac{1}{2} \! \leqslant \! n \!<\! 1$, 
or when $\frac{1}{3} \!<\!n \!<\! \frac{1}{2}$ and 
$e \!>\! e^*$. 
\end{theorem} 

The proof of Theorem~\ref{Thm2.2} is included in the proof 
for Theorem \ref{Thm2.3}.

\subsection{Codimension of Hopf bifurcation}

In this section, we consider the codimension of the 
Hopf bifurcation from the positive equilibrium ${\rm P_{1-}}$. 
Numerical examples  
have been given in \cite{LiZhouWuMa2007} to show that both stable 
and unstable limit cycles can bifurcate from Hopf critical points. 
This implies that the first focus value may become zero, and thus 
at least $2$ limit cycles can bifurcate due to the generalized Hopf 
bifurcation, leading to a codimension-$2$ Hopf bifurcation. 
We have the following theorem. 

\begin{theorem}\label{Thm2.3}
The system \eqref{Eqn11} has generalized Hopf bifurcation with 
codimension $2$, yielding maximal $2$ limit cycles enclosing an unstable 
equilibrium ${\rm P_{1-}}$, and the inner limit cycle is stable while 
the outer one is unstable. 
\end{theorem}

\begin{proof} 

[for Theorems \ref{Thm2.2} and \ref{Thm2.3}]
In order to derive the stability of ${\rm P_1}$ for studying the codimension 
of Hopf bifurcation, we solve $g=0$ for $e$, 
instead of $y_1$, to obtain 
\begin{equation}\label{Eqn20} 
e = \dfrac{ y_{1*} + \tfrac{1}{\kappa} - y_1}
{n\, (y_{\rm 1T}- y_1) (y_1- y_{1*})}. 
\end{equation} 
It follows from $e > 0$ that 
\begin{equation}\label{Eqn21} 
y_{1*} < y_1 < \min\Big\{y_{\rm 1T},\, y_{1*} +\dfrac{1}{\kappa} \Big\}.
\end{equation} 

The stability of equilibria is determined from the Jacobian of 
system \eqref{Eqn11}, given by 
$$
J(x,y) = \left[ \begin{array}{cc} 
\kappa\,[\,y-2x + e \, x (2y-3x)\, ] - (n+1) & \kappa\, x (1+ e\, x) \\ 
-1 & - n 
\end{array} 
\right]. 
$$
Evaluating this Jacobian at ${\rm P_1} \!=\! (1 \!-\! n y_1,y_1)$ yields 
the trace and determinant of the Jacobian $J({\rm P_1})$ as 
\begin{equation}\label{Eqn22} 
\begin{array}{cl} 
{\rm Tr}(J({\rm P_1})) \!\!\! & = 
\dfrac{ \kappa\, [(n+1) y_1-1]^2 + n+2-(n+1)^2 y_1 } {(n+1) y_1-1}, \\[2.0ex]  
\det (J({\rm P_1})) \!\!\! & = 
\dfrac{ (n+1)^3 (n y_1 -1)^2+n \,(\kappa-\kappa_{\rm T})\, [(n+1) y_1-1]^2 }  
{(n+1) y_1-1}.  
\end{array} 
\end{equation}
The determinant clearly shows that there exists a transcritical bifurcation 
between ${\rm P_0}$ and ${\rm P_1}$ at the critical point 
$\kappa=\kappa_{\rm T}$ at which $\det (J({\rm P_1}))=0$ leads to
$y_1 = \frac{1}{n}$. \\

Hopf bifurcation occurs when ${\rm Tr}(J({\rm P_1})) \!=\! 0$ and 
$\det (J({\rm P_1})) \!>\! 0$. Note that the formulas given in \eqref{Eqn22} 
are applicable for both solutions $y_{1+}$ and $y_{1-}$. 
In other words, the following results for ${\rm P_1}$ include both 
${\rm P_{1\pm}}\!: (x_{1\pm}, y_{1\pm})$. However, we know from the results 
in the previous section that Hopf bifurcation can only occur from 
the equilibrium ${\rm P_{1-}}$, since ${\rm P_{1+}}$ is a saddle point 
when it exists.  

A necessary condition for system \eqref{Eqn11} to undergo a Hopf 
bifurcation is ${\rm Tr}(J({\rm P_1})) =0$ which needs 
$ (n+2)-(n+1)^2 y_1 < 0$, that is then combined with the condition 
\eqref{Eqn21} to yield 
\begin{equation}\label{Eqn23} 
\dfrac{n+2}{(n+1)^2} < y_1 < \min\Big\{\dfrac{1}{n},\, 
\dfrac{1}{n+1} + \dfrac{1}{\kappa} \Big\} \quad \textrm{and} \quad 
\kappa < (n+1)^2. 
\end{equation} 

Solving ${\rm Tr}(J({\rm P_1})) = 0$ for $\kappa$ leads to the Hopf critical 
point, defined by 
\begin{equation}\label{Eqn24} 
\kappa_{\rm H} = \dfrac{(n+1)^2 y_1 -(n+2)}{ [ 1-(n+1) y_1]^2} 
= (n+1)^2 - 
\dfrac{ [2 (n+1)^2 y_1 - (2 n+3)]^2 
+ 3}{4 [ 1-(n+1) y_1]^2},  
\end{equation} 
which indicates that
$ \kappa_{\rm H} \! \in\! \big(0, (n+1)^2 \big)$. 
Hence, the equilibrium ${\rm P_1}$ is asymptotically stable if 
$\kappa \! \in \! (0,\kappa_{\rm H})$ 
(for which ${\rm Tr}(J({\rm P_1})) \!<\! 0$)
and unstable if $\kappa \! \in \! \big(\kappa_{\rm H}, (n+1)^2 \big)$ 
(for which ${\rm Tr}(J({\rm P_1})) \!>\! 0$). 
Hopf bifurcation occurs from ${\rm P_1}$ at the critical 
point $\kappa = \kappa_{\rm H}$. 
 
Next, in order to determine the stability of the bifurcating limit 
cycles, we need to compute the focus values. To achieve this, 
letting $\kappa \!=\! \kappa_{\rm H}$ and introducing 
the following affine transformation,  
$$
\left(\! \begin{array}{c} x \\ y \end{array} \! \right)
= \left(\!\! \begin{array}{c} 1 - n y_1 \\ y_1 \end{array} 
\!\! \right) + \left[\! \begin{array}{cc}
1 & 0 \\[0.5ex] 
\dfrac{n [(n + 1) y_1 - 1]}{(n + 1) (n y_1 - 1)} 
& \dfrac{-\omega_c [(n + 1) y_1 - 1]}{(n+1) (n y_1 -1)} \\ 
\end{array}
\! \right] \!  
\left(\! \begin{array}{c} u \\ v \end{array} \! \right), 
$$
where
$$ 
\omega_c = \sqrt{ \dfrac{n^2+n+1 - n (n+1)^2 y_1 }{(n+1) y_1-1}}, 
$$
into \eqref{Eqn11} we obtain the following system, 
\begin{equation}\label{Eqn25} 
\begin{array}{rl}  
\dfrac{{\rm d}u}{{\rm d} \tau} =\!\!\! & \omega_c \, v 
- \dfrac{1}{(n+1) Q_1^2 Q_2^2} \big( Q_1 u^2 
- \omega_c Q_1 Q_2 Q_3\, uv 
+ u^3 - \omega_c Q_2\, u^2 v \big), 
\\[2.0ex]  
\dfrac{{\rm d}v}{{\rm d} \tau} =\!\!\! & 
-\, \omega_c \, u 
- \dfrac{n}{(n+1) Q_1^2 Q_2^2\, \omega_c} \big( Q_1 u^2  
- \omega_c Q_1 Q_2 Q_3\, uv 
+ u^3 
- \omega_c Q_2\, u^2 v \big), 
\end{array} 
\end{equation} 
whose linear part is in the Jordan canonical form, where 
$$
Q_1 = 1-n y_1, \quad Q_2 = (n+1) y_1 -1, \quad Q_3 = (n+1)^2 y_1 - n. 
$$
Note that $\omega_c>0$ requires that 
\begin{equation}\label{Eqn26} 
y_1 < \dfrac{n^2+n+1}{n(n+1)^2} = \dfrac{1}{n} - \dfrac{1}{(n+1)^2}.  
\end{equation}
As a matter of fact, ${\rm P_1}$ is a saddle point when 
$y_1  > \frac{1}{n} - \frac{1}{(n+1)^2}$, and a B-T bifurcation point
when $y_1 = \frac{1}{n} - \frac{1}{(n+1)^2}$. \\[1.0ex]  

Combining the condition $y_1 > \frac{n+2}{(n+1)^2} $ 
in \eqref{Eqn23} and that in \eqref{Eqn26} yields that 
$ \frac{n+2}{(n+1)^2} <  \frac{1}{n} - \frac{1}{(n+1)^2}$, 
which in turn gives $n <\! 1$.
Therefore, it follows from \eqref{Eqn23} and \eqref{Eqn26} 
that the restriction conditions on $y_1$ and $\kappa$ are given by 
$$
y_{\rm 1L} < y_1  < y_{\rm 1U}, 
$$
where
\begin{equation}\label{Eqn27} 
y_{\rm 1L} = \dfrac{n+2}{(n + 1)^2}, \quad
y_{\rm 1U} = \min\Big\{ \dfrac{1}{n}-\dfrac{1}{(n + 1)^2},\ 
\dfrac{1}{n + 1} + \dfrac{1}{\kappa_{\rm H}} \Big\}, \ \ (0<n<1), 
\end{equation} 
under which $\,Q_i \!>\! 0, \, i=1,2,3$, and 
$n \!+\! 1 \!-\! n\, Q_3 \!>\!0$. 

Now, applying the Maple program \cite{Yu1998} for computing the normal 
form of Hopf and generalized Hopf bifurcations to system \eqref{Eqn25},
we obtain the following focus values, 
\begin{equation}\label{Eqn28} 
\hspace{-0.10in}
\begin{array}{rl}
v_1 = \!\!\! & \dfrac{v_{\rm 1a} }{8 (n+1) Q_1 
Q_2^3\, [n+1 -n Q_3]},
\\[2.5ex]
v_2 = \!\!\! & \dfrac{-\, v_{\rm 2a}}{192 (n+1)^3 Q_1^3
Q_2^6\, [n+1 -n Q_3]^3 }, 
\\[2.5ex] 
v_3 = \!\!\! & \dfrac{ v_{\rm 3a}}{18432 (n+1)^5 Q_1^5
Q_2^9\, [n+1 -n Q_3]^5 }, 
\\[0.0ex]
\vdots \! &  \vspace{-0.00in}
\end{array}
\end{equation} 
showing that $v_i$ and $v_{i {\rm a}} $ ($i=1,2,3$) have the same sign. 
Here, $ v_{\rm 1a},\, v_{\rm 2a}, \, \cdots$ 
are polynomials in $ n $ and $y_1$. In particular,
\begin{equation}\label{Eqn29} 
\begin{array}{rl} 
v_{\rm 1a} = \!\!\! & 
n (n+1)^4 y_1^2 - 2 (n^2+n+1) (n+1)^2 y_1 + (n^3+2 n^2+5 n+3),
\\[1.0ex] 
v_{\rm 2a} = \!\!\! & 
n^3 (n \!+\! 1)^{12} (28 n \!+\! 27) y_1^7
\!-\! n^2 (n \!+\! 1)^{10} (196 n^3 \!+\! 350 n^2 \!+\! 267 n \!+\! 108) y_1^6
\\[0.5ex] 
&+\,n (n+1)^8 (588 n^5+1533 n^4+2101 n^3+1683 n^2+672 n+135) y_1^5 
\\[0.5ex] 
&-\, (n+1)^6 (980 n^7+3360 n^6+6500 n^5+7711 n^4+5620 n^3 +2567 n^2
\\[0.5ex] 
& +\, 602 n+54) y_1^4 +(n+1)^4 (980 n^8+4165 n^7+10330 n^6+16174 n^5 
\\[0.5ex] 
& +\, 16730 n^4 +11680 n^3+5215 n^2+1400 n+169) y_1^3 
\\[0.5ex] 
& -\, (n+1)^2 (588 n^9+2982 n^8+8995 n^7+17466 n^6+23190 n^5
\\[0.5ex] 
&+\, 21409 n^4+13397 n^3+5563 n^2+1443 n+192) y_1^2
\\[0.5ex] 
&+\, (n+1) (196 n^9+959 n^8+3138 n^7+6349 n^6+8961 n^5+8410 n^4
\\[0.5ex]
&+\, 4886 n^3+1470 n^2+84 n-36) y_1 -28 n^9-132 n^8-474 n^7
\\[0.5ex] 
&-\,979 n^6-1470 n^5-1291 n^4-460 n^3 +334 n^2+517 n+189.
\end{array}
\end{equation}

To find the maximal number of limit cycles bifurcating from the
Hopf critical point, we need to find the solutions satisfying $v_k=0$ as
many as possible. Since there are only two free parameters
($n,y_1$) in $v_k$, the maximal number of bifurcating limit cycles can be
three. However, in solving practical problems, due to the constraints on 
the parameters, this maximal number is usually not reachable.
Now, let us first consider the possibility of $3$ limit cycles.
Eliminating $y_1$ from the two equations,
$v_{\rm 1a} = v_{\rm 2a} = 0$, yields a solution for $y_1=\bar{y}_1$, where 
\begin{equation}\label{Eqn30} 
\bar{y}_1 = \dfrac{ 23+(1-2 n) \big[ (1-2 n) (1+14 n+24 n^2)+46 n^3 \big] }
{ 4 (4+n)+ 2 n (1-2 n) [(1-2 n) (6+15 n+28 n^2)+55 n^3] }, 
\end{equation} 
and a resultant equation:
\begin{equation}\label{Eqn31} 
\textrm{Res}\, (n) = n (2 n+1) (3 n^2-11 n+11) (2 n-1) =0,
\end{equation} 
which has only one positive solution $n \!=\! \frac{1}{2}$ for which
$ \bar{y}_1 \!=\! \frac{23}{18}$, 
yielding $v_1 \!=\! \frac{3240}{17303} \! \ne \! 0$.  
Hence, bifurcation of $3$ limit cycles 
is not possible, since it requires that $v_1 = v_2 = 0$.

The next best possibility is bifurcation of $2$ limit cycles,
which needs $v_1=0$. The discriminant of quadratic polynomial $v_{\rm 1a}$ is
$$
\Delta_1 = 4 (1-2n) (n+1)^5.
$$
Thus, $\Delta_1 \!<\! 0$ when $n \!>\! \frac{1}{2}$, yielding 
$v_{\rm 1a} \!>\! 0$ and so $v_1 \!>\! 0$. 
When $n \!=\! \frac{1}{2}$, the condition 
$ \frac{n+2}{(n+1)^2} \!<\! y_1 \!<\! \frac{1}{n} \!-\! 
\frac{1}{(n+1)^2} $ is reduced to 
$\frac{10}{9} < y_1 < \frac{14}{9}$, and then 
$v_1$ becomes 
$$ 
v_1 = \dfrac{14-9 y_1}{3 m^2 (2-y_1) (3y_1-2)^3} > 0.  
$$
The above result indicates that when $\frac{1}{2} \! \leqslant \! 
n \!<\! 1$, the Hopf bifurcation is subcritical and bifurcating 
limit cycles are unstable. When $n < \frac{1}{2}$,
solving the quadratic polynomial equation
$v_{\rm 1a}=0$ yields two real positive solutions,
\begin{equation}\label{Eqn32} 
y_{1 \pm} = \dfrac{1}{n (n+1)^2}
\left[ n^2+n+1 \pm \sqrt{ (1-2 n) (n+1)} \right]. 
\end{equation} 
Hence, $v_{\rm 1a} < 0$ for $ y_1 \in (y_{1-}, y_{1+})$ and 
$v_{\rm 1a} > 0$ for $ y_1 \in (0, y_{1-}) \bigcup (y_{1+}, \infty)$. 

Moreover, it is easy to show that \vspace{-0.05in}
$$
y_{1-} < \dfrac{1}{n} - \dfrac{1}{(n+1)^2} < y_{1+} ,
$$
and 
$$
y_{1-} \gtrless y_{\rm 1L} = \dfrac{n+2}{(n+1)^2}
\ \ \Longleftrightarrow \ \ \sqrt{(1 - 2 n)(n + 1)}
\lessgtr 1 - n \ \ \Longleftrightarrow \ \ n (3n - 1) \gtrless  0. 
$$
Therefore, for $n \! \leqslant \! \frac{1}{3}$, 
the condition $y_{1-} \! \leqslant \! y_{\rm 1L}$ implies that any feasible 
values of $y_1$ for Hopf bifurcation satisfy
$$ 
y_{1-} \leqslant  y_{\rm 1L} <  y_1 < y_{\rm 1U}  
\leqslant \dfrac{1}{n} - \dfrac{1}{(n+1)^2} < y_{1+}, 
$$ 
and so $v_{\rm 1a} \!<\! 0$, that is, 
$v_1 \!<\! 0$. Hence, the Hopf bifurcation is supercritical and 
bifurcating limit cycles are stable for $0 \!<\! n \!\leqslant \! 
\frac{1}{3}$. Summarizing the above results, we have shown that 
one limit cycle can be generated from Hopf bifurcation for 
$ n \! \in \! \big(0,\frac{1}{3} \big] \bigcup \big[\frac{1}{2},1 \big)$, and 
the Hopf bifurcation is supercritical if 
$ n \! \in \! \big(0,\frac{1}{3} \big]$, and subcritical 
if $ n \! \in \! \big[\frac{1}{2},1 \big)$.  
The conditions under which two limit cycles can occur from a Hopf bifurcation 
are given by 
$$
\dfrac{1}{3} < n < \dfrac{1}{2}, \quad 
y_1 = y_{1-} \ \ \Longrightarrow \ \ v_1 = 0, \ \ \textrm{and thus} \ \ 
y_1 \gtrless y_{1-} \ \ \Longleftrightarrow \ \ v_1 \lessgtr 0. 
$$
Further, it can been shown by using \eqref{Eqn27} 
that for $\frac{1}{3} \!<\! n \!<\!  \frac{1}{2}$, the following holds: 
$$ 
\begin{array}{rl} 
& \dfrac{1}{n} - \dfrac{1}{(n+1)^2} < \dfrac{1}{n+1} 
+ \dfrac{1}{\kappa_{\rm H}} 
\\[2.5ex] 
\Longleftrightarrow &  
 \dfrac{1}{n} - \dfrac{1}{(n+1)^2} - \dfrac{1}{n+1} 
- \dfrac{1}{\kappa_{\rm H}}<0
\\[2.5ex]
\Longleftrightarrow & 
 - \, \dfrac{n (n+1)^4 y_1^2-(2 n^2+2 n+1) (n+1)^2 y_1  
        +n^3+2 n^2+2 n+2} {n (n+1)^2 [(n+1)^2 y_1-(n+2)]} < 0, 
\end{array} 
$$
because the discriminant of the numerator of the left-hand 
side of the last inequality in the above is equal to $-(n\!+\!1)^4 (4n\!-\!1) 
\!<\! 0 $ for $ n \! \in\! (\frac{1}{3},\frac{1}{2})$. 
This implies that $ y_{\rm 1U} \!=\! \frac{1}{n} \!-\! \frac{1}{(n+1)^2}$ 
and $y_{\rm 1L} \!<\! y_{1-} \!<\! y_{\rm 1U} \!<\! y_{1+}$ 
for $ n \! \in \! (\frac{1}{3},\frac{1}{2})$. 
Consequently, when $n \! \in \! \big(\frac{1}{3},\frac{1}{2}\big)$, 
the Hopf bifurcation is supercritical for $y_1 \! \in \! (y_{1-},y_{\rm 1U})$ 
($v_1 \!<\! 0$), and subcritical for $y_1 \! \in \! (y_{\rm 1L},y_{1-})$ 
($v_1 \!>\!0$). 

To transform the above stability expression in $y_1$ back to that 
in the parameter $e$, we substitute 
$\kappa \!=\! \kappa_{\rm H}$ and $y_1$ into $e$ in \eqref{Eqn20} to obtain 
$$ 
e(y_1) = \dfrac{1}{(1- n y_1) [(n+1)^2 y_1 -(n+2)]} >0, \quad 
\textrm{for} \ \ y_{\rm 1L} < y_1 < y_{\rm 1U} .  
$$  
Then, substituting $y_1=y_{1-}$ into $e(y_1) $ yields 
$e^* \!=\! e(y_{1-})$. Further, we can show that 
$$
\dfrac{{\rm d}\, e (y_1)}{{\rm d} y_1} = 
\dfrac{2 n (n \!+\! 1)^2 y_1 \!-\! (2 n^2 \!+\! 4 n \!+\! 1)}
      {(1\!-\! n y_1)^2 [(n \!+\! 1)^2 y_1 \!-\! (n \!+\! 2)]^2} 
\left\{\!\! 
\begin{array}{ll}
< 0, & \! \textrm{for} \ \ y_1 \!<\! y_{\rm 1 min}, \\[0.0ex] 
> 0, & \! \textrm{for} \ \ y_1 \!>\! y_{\rm 1 min}, 
\end{array} 
\right.  \ \ y_{\rm 1 min} = \dfrac{2n^2 \!+\! 4n \!+\! 1}{2n(n+1)^2}.   
$$
It is easy to prove that $y_{\rm 1L} \!<\! y_{\rm 1 min} \!<\! 
y_{\rm 1U}$, and $ y_{1-} \!<\! y_{\rm 1 min}$. 

In addition, a direct computation leads to that 
$$
\begin{array}{rl}  
e^* - e(y_{\rm 1U}) = e^* - e_3 
= \dfrac{ (1 \!+\! n)^2 (1 \!-\!2 n) \big[(1\!+\!n) (5n\!-\!2) 
\!-\! (1 \!-\! n) \sqrt{(1 \!+\! n) (1 \!-\! 2n)}
\, \big]} {(1-n) (3 n-1) (3 n^2+n-1)}>0
\end{array}
$$ 
for $\frac{1}{3} \!<\! n \!<\! \frac{1}{2}$. 
This clearly indicates that 
$e \!<\! e^*$ (respectively $ e \! >\! e^*$) when 
$y_1 \!>\! y_{1-}$ (respectively $ y_1 \! <\! y_{1-}$).  
Hence, for $n \! \in \! \big(\frac{1}{3},\frac{1}{2}\big)$,
the Hopf bifurcation is supercritical (respectively subcritical) 
if $e \!<\! e^*$ (respectively $e \!>\! e^*$).  
This finishes the proof for the first part of Theorem \ref{Thm2.2}. 
   
Now, based on the above established results, we consider the Hopf bifurcations 
in Lemma~\ref{Lem1}. 
For the Case (I-b), it is obvious that the two Hopf bifurcations are 
supercritical since $n \! \leqslant \! \frac{\sqrt{2}-1}{2} \!<\! \frac{1}{3}$, 
and bifurcating limit cycles are stable.  
For the Case (II-c)(i), with the condition $n\!<\! \frac{1}{2}$ 
and $e_3 \!<\! e \!<\! e_4 $, we know that 
the Hopf bifurcation is supercritical for $0 \!<\! n \! \leqslant \! 
\frac{1}{3}$. When $ \frac{1}{3} \!<\! n \!<\! \frac{1}{2}$, 
similar to prove $e_3 \!<\! e^*$, we can show that 
$e_4 \!<\! e^*$, which yields $e \!<\! e^*$. 
Thus, the Hopf bifurcation is also 
supercritical for $ \frac{1}{3} \!<\! n \!<\! \frac{1}{2}$.  
For the Case (II-c-ii) with the condition $n\!<\! 1$ and 
$e \!>\! \max\{e_3,e_4\}$, 
the Hopf bifurcation is supercritical for $0 \!<\! n \!\leqslant \! 
\frac{1}{3}$, and subcritical for $\frac{1}{2} \!\leqslant \! n 
\!<\! 1$. When $\frac{1}{3} \!<\! n \!<\! \frac{1}{2}$, it has been shown 
in the above that $ e \!>\! \max\{e_3,e_4\}$,
the Hopf bifurcation is supercritical for 
$\max\{e_3,e_4\} \!<\! e \!<\! e^*$, and subcritical for 
$ e \!>\! e^*$.

This finishes the proof for Theorem \ref{Thm2.2}.   

To prove Theorem~\ref{Thm2.3}, note that 
we have shown that $v_1 \!=\! 0$ at $y_1 \!=\! y_{1-}$ (or at 
$e \!=\! e^*$), indicating that two limit cycles 
can bifurcate from the Hopf critical point $\kappa \!=\! \kappa_{\rm H}$ 
if $e \!=\! e^*$. To complete the proof for 
Theorem \ref{Thm2.3}, we need to verify $v_2 \! \ne \!0$ when 
$v_1 \!=\! 0$ (i.e., at $y_1 \!=\! y_{1-}$ 
or $ e \!=\! e^*$). 
Substituting the solutions $y_{1-}$ into $v_2$ we obtain 
$$
\left. v_2 \right|_{v_1=0}
= \dfrac{-\, \tilde{v}_2}{48 n^3 (1 \!-\! 2 n) (n \!+\! 1)^9
   [\sqrt{(1 \!-\! 2 n) (n \!+\! 1)} \!+\! n]^6
   [\sqrt{(1 \!-\! 2 n) (n \!+\! 1)} \!-\! 1]^6}, 
$$
where
$$
\begin{array}{rl}
\tilde{v}_2 = \!\!\! &
[(1-2 n) (8+8 n+15 n^2)+32 n^3] \sqrt{ (1-2 n) (n+1)}
\\[1.0ex]
& -\, [ (1-2 n) (8+4 n+2 n^2+11 n^3)+20 n^4],
\end{array}
$$
whose sign is the same as that of 
$$
\begin{array}{rl}
\tilde{\tilde{v}}_2 = \!\!\! &
[(1-2 n) (8+8 n+15 n^2)+32 n^3]^2 (1-2 n) (n+1)
\\[1.0ex]
& -\, [ (1-2 n) (8+4 n+2 n^2+11 n^3)+20 n^4]^2
\\[1.0ex]
= \!\!\! & -\, n^4 (2 n+1)^2 (3 n^2-11 n+11) < 0,
\end{array}
$$
yielding $\tilde{v}_2 \!<\! 0$, that is, $\left. v_2 \right|_{v_1=0} \!>\! 0$.
This shows that the codimension of the Hopf bifurcation is indeed two.
Moreover, the outer bifurcating limit cycle is unstable and the inner one
is stable, and both them enclose an unstable focus ${\rm P_{1-}}$.

Finally, we want to show that the generalized Hopf bifurcation, leading 
to bifurcation of $2$ limit cycles under the condition 
$\frac{1}{3} \! < \! n \! < \frac{1}{2}$, can 
only occur in Case (II-c-ii). First, it is easy to see that it is not possible 
for the Case (I-b) since it requires $n \!< \! \frac{\sqrt{2}-1}{2}$. 
Then, with the condition $\frac{1}{3} \! < \! n \! < \frac{1}{2}$, 
we can show that $e \!>\! e_3$ and $e \!>\! e_4$. 
By using \eqref{Eqn20} with $y_1 \!=\! y_{1-}$ given in \eqref{Eqn32} 
and $\kappa = \kappa_{\rm H}$ given in \eqref{Eqn24}, we obtain 
$$ 
e = \dfrac{(n+1)^2 [1-n+ \sqrt{(1-2n) (n+1)}]}
                 {(3n-1) [n + \sqrt{(1-2n) (n+1)}]}.  
$$ 
Then, a direct computation shows that 
$$ 
\begin{array}{ll}  
e - e_3 = \dfrac{(n+1)^2 (1-2 n) [1+n+2 \sqrt{(1-2 n) (n+1)}]}  
        {(1-n) (3n-1) [n+ \sqrt{(1-2n)(n+1)]}} > 0,  
\\[2.5ex] 
e - e_4 
= \dfrac{(n+1)^2 (1-2 n) [2n+ \sqrt{(1-2 n) (n+1)}]}  
        {n (3n-1) [n+ \sqrt{(1-2n)(n+1)]}} > 0, 
\end{array}
$$ 
for $\frac{1}{3} \! < \! n \! < \frac{1}{2}$. 
This indicates that $2$ limit cycles cannot occur in Cases (II-c-i) 
and (II-d), and can only bifurcate in 
Case (II-c-ii), for which $ \varepsilon \!>\! \max\{\varepsilon_3,
\varepsilon_4\}$, and it can be shown that 
$ \kappa_{\rm SN} \!<\! \kappa_{\rm H_+} \!<\! \kappa^*$ as follows:
$$ 
\begin{array}{rl}
\kappa_{\rm H_+} - \kappa_{\rm SN} 
\!\!\! & = \dfrac{ 2 e n (1-n^2)^2 (e-e_3)^2}
             {(e+ e_1)^2 \{ e [n (e+ e_4)-e_2] +n (e-e_4)
              \sqrt{e (e-e_2)} \}} >0, \\[2.5ex]  
\kappa^* - \kappa_{\rm H_+} 
\!\!\! & = \dfrac{ 2 e n (1+n) (1-n^2)^2 (e-e_3) (e-e_4)}
             {(e+ e_1)^2 \{e (e- e_4) +n (e-e_4) \sqrt{e (e-e_2)} \}} >0, 
\end{array} 
$$ 
since for this case, we have 
$e > e_4 > e_3 > e_2 > e_1 $, and 
$$ 
n(e + e_4) - e_2 > 2n e_4 - e_2 = 2(1-2n)(n+1)^2 > 0. 
$$  

The proof for Theorem~\ref{Thm2.3} is complete. 
\end{proof}

\subsection{Simulations} 

Simulations for the stable limit cycles in the Cases (I-b), (II-c-i) and 
(II-d), as well as an unstable limit cycle in the Case (II-c-ii) with 
$n \!=\! \frac{3}{4} \!\in\! \big[ \frac{1}{2},1\big)$ have been shown 
in \cite{ZengYu2021}. 
In \cite{LiZhouWuMa2007}, the authors used system \eqref{Eqn8} 
to give four simulations to show 
the existence of bifurcating limit cycles, among them three are stable, and 
one is unstable. The parameter values for these simulations are given below, 
with our formula $v_1$ in \eqref{Eqn28} and \eqref{Eqn29}
to confirm the stability (with the bold-faced numbers for the values of $n$). 
In order to keep consistent with the results given in \cite{LiZhouWuMa2007}, 
we also use \eqref{Eqn8} for our simulation, which will give the 
same results obtained by using \eqref{Eqn11} under the transformation 
\eqref{Eqn10}. 
$$
\!\! \begin{array}{llll}
\textrm{Stable LC}: &\!\! (m,n,\varepsilon,k) =&\!\!\!
(1.6,\, {\bf 0.30},\, 1.50,\, 0.25), & v_1 = -\,0.082227,  
\\[1.0ex]
&\!\! &\!\!\! (2.0,\, {\bf 0.25},\, 1.50,\, 0.15625), & v_1 = -\,0.038194, 
\\[1.0ex]
&\!\! &\!\!\! (1.5,\, {\bf 0.30},\, 1.55,\, 0.25),& v_1 = -\,0.135503,   
\\[1.0ex]
\textrm{Unstable LC}: &\!\! (m,n,\varepsilon,k) =&\!\!\!
(11.07825,\, {\bf 0.4771},\, 0.995,\, 0.0295), & v_1 = +\, 0.001773.  
\end{array}
$$

\begin{figure}[!h]
\vspace*{-1.90in} 
\begin{center}
\hspace*{-0.12in} 
\begin{overpic}[width=1.00\textwidth,height=0.941\textheight]{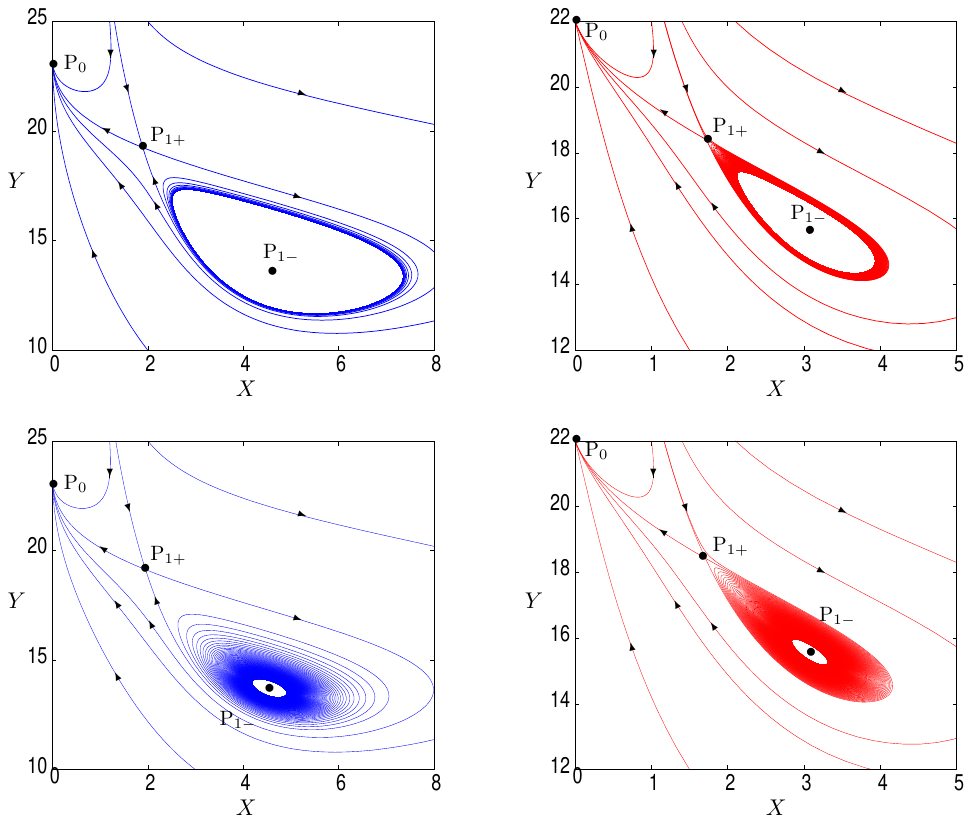} 
\end{overpic}

\vspace{-2.20in} 

\caption{Simulated phase portraits for the Case (II-c-ii) of 
the model \eqref{Eqn8} 
with $n\!=\! 0.4771$: (a) an unstable limit cycle 
for $m\!=\! 11.07825$, $\varepsilon \!=\! 0.995  $ 
and $ k \!=\! 0.0295 $ \cite{LiZhouWuMa2007}; (b) a stable limit cycle for 
$m\!=\! 10.5$, $\varepsilon \!=\! 0.43 $ and $k\!=\! 0.0509 $;
(c) an unstable limit cycle 
for $m\!=\! 11.07825$, $\varepsilon \!=\! 0.995  $ 
and $ k \!=\! 0.029225  $; and (d) a stable limit cycle for 
$m\!=\! 10.5$, $\varepsilon \!=\! 0.43  $ and $k\!=\! 0.05094 $.} 
\label{fig2} 
\end{center}
\end{figure}

It is not surprising to see that the first three values of $n$
are less than $\frac{1}{3}$, yielding stable limit cycles. 
For the last example, $n\!=\! 0.4771 \! \in \! \big(\frac{1}{3},
\frac{1}{2} \big)$, and using the formula \eqref{Eqn19}, we obtain 
$\varepsilon^* \!=\! 0.485004 \!<\! 
\varepsilon \!=\! 0.995 $, yielding $v_1 \!=\! 0.001773$ and thus 
the Hopf bifurcation is subcritical and the bifurcating limit cycle is 
unstable, as shown in Figure~\ref{fig2}(a). 
However, for this case when $ n \in (\frac{1}{3},\frac{1}{2})$,  
the Hopf bifurcation is not necessarily subcrtical. 
Parameters can be changed to get a supercritical Hopf bifurcation. 
For example, keeping $n\!=\! 0.4771$ unchanged and the following 
values for other parameter values, we obtain 
$$
\textrm{Stable LC}: \ \ (m,n,\varepsilon,k) = 
(10.5,\, {\bf 0.4771},\, 0.43,\, 0.0509), 
$$ 
which yields $\varepsilon^* \!=\! 0.511714 \!>\! 
\varepsilon = 0.43 $ and $v_1 \!=\! -0.000831$, implying that 
the Hopf bifurcation is supercritical and the bifurcating limit cycle is 
stable, see Figure~\ref{fig2}(b). 
It can be seen that the limit cycles shown in Figures~\ref{fig2}(a) and 
\ref{fig2}(b) are pretty large and so the Hopf bifurcation prediction 
for the amplitude based on the normal forms might be not applicable. 
The Hopf critical value for Figures~\ref{fig2}(a) and 
\ref{fig2}(b) are $k_{\rm H} \approx 0.029221$ 
and $k_{\rm H} \approx 0.050942$, respectively, which yields 
the perturbations $\mu \approx 0.000279$ ($\frac{\mu}{k_{\rm H}}=0.96 \%$) 
and $\mu \approx 0.000042$ ($\frac{\mu}{k_{\rm H}}=0.08 \%$) 
for Figures~\ref{fig2}(a) and \ref{fig2}(b), respectively.    
This indicates that the perturbations for both cases are actually 
small, but they yield the approximation for the amplitudes of the 
two small limit cycles as $ r= 2.4734 $ and $ r = 0.8514$, respectively. 
They are pretty large, but it is interesting that the normal forms 
still predict their correct stability. For a comparison, we present 
another two simulations depicted in Figures~\ref{fig2}(c) and \ref{fig2}(d), 
corresponding to $k=0.029225$ 
(giving $\mu \approx 0.000004$ and $\frac{\mu}{k_{\rm H}}=0.015 \%$)
and $k=0.05094$   
(giving $\mu \approx 0.000002$ and $\frac{\mu}{k_{\rm H}}=0.005 \%$),
respectively. Both bifurcating limit cycles are quite small, with 
unstable and stable ones for the former and latter cases, respectively. 
The normal forms for these two cases predict the amplitudes for the 
two limit cycles as $r \approx 0.3142$ and $r \approx 0.2055$, respectively,
showing that Hopf bifurcation theory is perfectly applicable.   

For the two limit cycles arising from 
Hopf bifurcation, there exists an infinite set of parameter values. 
For example, we first choose $m=2$ since it is free, and then take 
$$
n=\frac{5}{11} \in \Big(\dfrac{1}{3},\, \dfrac{1}{2} \Big), 
$$ 
yielding the values for the critical point, 
$$ 
y_1 = y_{1-} = \dfrac{Y_{1-}}{2} = \dfrac{1727}{1280}, \quad 
k_{\rm H} = \dfrac{1280}{5929}, \quad  
\varepsilon^* = \dfrac{320}{99},  
$$ 
for which $v_0 \!=\! v_1 \!=\! 0$, $v_2 \!=\! 
\frac{46661632000}{845676707337}$. 
\vspace{0.03in} It is easy to verify that the above parameter values 
belong to the case (II-c-ii) in Lemma~\ref{Lem1}. Only one Hopf critical 
point exists in this case.  

Then, perturbing $y_1$ and $k$ as follows:  
$$ 
y_1=y_{1-} + \epsilon_1, \quad 
k=k_{\rm H} + \epsilon_2, \quad \textrm{where} \ \  
\epsilon_1 = \dfrac{1}{100}, \ \ \epsilon_2 = -\,\dfrac{1}{50000}, 
$$ 
we obtain 
$$ 
k = \dfrac{203821518599}{924070050000}   
\quad \textrm{and} \quad 
\varepsilon = \dfrac{3407490109063040}{1096763591581219}. 
$$ 
For these perturbed parameter values, we have the normal form for 
the amplitude of oscillation up to $5$th-order, given by 
$$ 
\!\!\! \begin{array}{rl} 
\dot{r} \!\!\! &= r \, \big( v_0 + v_1 r^2 + v_2 r^4 \big) 
\\[1.0ex] 
&= r \big ( \tfrac{4299}{220000000} 
\!-\! \tfrac{216980684800000}{83938145042658897}\, r^2 
\!+\! \tfrac{14817759528898477514254458827898880000000}
{200973840260810036901676626005378422866729} \, r^4 \big). 
\end{array}  
$$ 
Setting $\dot{r}=0$ gives the approximate solutions for the 
amplitudes of the two limit cycles: 
$$ 
r_1 \approx  0.105015, \quad r_2 \approx 0.155024. 
$$ 
The simulation of this numerical example is shown in Figure~\ref{fig3}.
It is seen from this figure that the amplitudes of the simulated two 
limit cycles agree very well with the analytical predictions given above.

\begin{figure}[!h]
\vspace{-1.80in} 
\begin{center}
\hspace*{-0.10in} 
\begin{overpic}[width=1.06\textwidth,height=0.941\textheight]{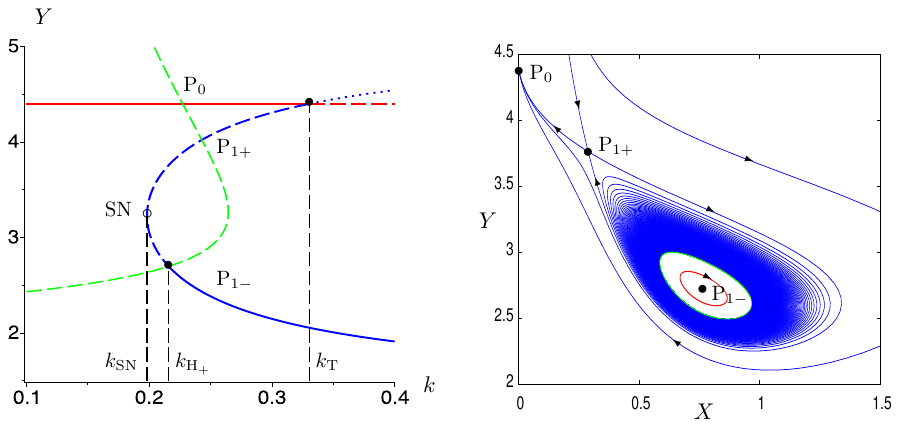} 
\end{overpic}

\vspace{-3.90in} 

\caption{(a) Bifurcation diagram for the epidemic model \eqref{Eqn8} 
projected on the $k$-$y_1$ plane 
for $m=2$, $n=\frac{5}{11}$ and $\varepsilon \!=\! \frac{320}{99}$,  
corresponding to the Case (II-c-ii) 
in Lemma~\ref{Lem1} having one Hopf critical 
point, with $k_{\rm SN} \!=\! \tfrac{57600}{290521}$, 
$k_{\rm H_+} \!=\! \frac{1280}{5929} $ and
$k_{\rm T} \!=\! \frac{40}{121}$; and    
(b) simulation of $2$ limit cycles for  
perturbed values $ k = \frac{203821518599}{924070050000} $ 
and $\varepsilon=\frac{3407490109063040}{1096763591581219}$,  
with the outer one unstable 
(in green color) and inner one stable (in red color).}
\label{fig3} 

\end{center}
\end{figure}

\section{B-T bifurcation of system \eqref{Eqn11}}

In this section, we present an analysis on the B-T bifurcation 
of system \eqref{Eqn11}. In \cite{LiZhouWuMa2007}, the classical method 
was used to find the condition for codimension-$2$ B-T bifurcation 
for system \eqref{Eqn8}. 
In \cite{LiLiMa2015}, the six-step transformation approach was applied to 
analyze the codimension-$3$ B-T bifurcation of system \eqref{Eqn8}. 

\subsection{Determining the codimension of B-T bifurcation}

Here, we apply the SNF theory \cite{Yu1999,YL2003,YuZhang2019} to determine 
the codimension of B-T bifurcation. 
To achieve this, we first use $\kappa$ to solve the equation $g=0$, 
where the function $g$ is given in \eqref{Eqn14}, and then use 
$y_1$ and $e$ to solve ${\rm Tr}(J({\rm P_1}))
=\det (J({\rm P_1})) = 0$ to obtain
\begin{equation}\label{Eqn33} 
\kappa = n(1 - n)(1+n)^2, \quad 
e = \dfrac{(1+n)^2}{1 \!-\! n}, \quad 
x_1 = \dfrac{n}{(1+n)^2}, \quad 
y_1 = \dfrac{n^2 + n +1 )}{n(1+n)^2}, 
\end{equation} 
which requires $0<n<1$. Then, introducing the affine transformation, 
\vspace{0.10in}
\begin{equation}\label{Eqn34} 
\left(\begin{array}{c} x \\ y \end{array} \right)
= \left(\begin{array}{c} \dfrac{n}{(1+n)^2} \\[2.0ex] 
\dfrac{n^2+n+1}{n(1+n)^2} \end{array} \right)
+ \left[ \begin{array}{rc} n & 1 \\ -1 & 0 \end{array} \right]
\left(\begin{array}{c} u \\ v \end{array} \right)
\end{equation} 
into \eqref{Eqn11} yields
\begin{equation}\label{Eqn35}
\!\!\!\! \begin{array}{rl}
\dfrac{{\rm d} u}{{\rm d} \tau} = \!\!\! & v, \\[2.0ex]
\dfrac{{\rm d} v}{{\rm d} \tau} = \!\!\! &
- (n+1)^3 (nu+v) [n^2 u + (n - 1) v]
- n (n+ 1)^4 (n u+v)^2 [(n+1)u+v]. 
\end{array}
\end{equation}
Next, applying the $5$th-order change of variables:
$$
\begin{array}{rl}
u =\!\!\!& y_1 - \dfrac{(n \!+\! 1)^2 (2 n^2 \!-\! 1)}{4}\, y_1^2
- \dfrac{(n\!+\! 1) (n^2 \!+\! n \!+\!1) 
(2 n^2 \!-\! 3 n \!-\! 1) (2 n^2 \!+\! 4 n \!+\! 1)}{12 n^2 }\, y_1 y_2
\\[2.0ex]
& -\, \dfrac{(n+1) (20 n^5-80 n^3+16 n^2+5 n+4)}{240 n^4}\, y_2^2
+ \cdots 
\\[2.0ex]
v = \!\!\! & y_2
- (n-1) (n+1)^3\, y_1 y_2
- \dfrac{(n+1) (2n^2-3n-1)(2n^2+4n+1)}{12 n^2 }\, y_2^2 + \cdots
\end{array}
$$
and the time scaling,
$$
\tau = \Big[ 1 + \dfrac{(n+1)^2}{2 m} y_1 + t_{30}\, y_1^3 \Big]\, \tau_1, 
$$ 
where $t_{30}$ is a function in $n$, 
into \eqref{Eqn11} yields the SNF up to $5$-th order as follows:
\begin{equation}\label{Eqn36}
\begin{array}{ll} 
\dfrac{{\rm d} y_1}{{\rm d} \tau_1} = y_2, \\[2.0ex] 
\dfrac{{\rm d} y_2}{{\rm d} \tau_1} = c_{20}\, y_1^2 + c_{11}\, y_1 y_2
+ c_{31}\, y_1^3 y_2 + c_{41}\, y_1^4 y_2,
\end{array}
\end{equation}
in which 
$$
\begin{array}{rl}
c_{20} =\!\!\!&-\,n^3 (n+1)^3, \\[1.0ex] 
c_{11} =\!\!\!&-\,n (n+1)^3(2n-1), \\[1.0ex]
c_{31} =\!\!\!&-\,\dfrac{(n+1)^6 (40 n^5+44 n^4-18 n^3+9 n^2-7 n+2)}{40}, 
\\[1.0ex] 
c_{41} = \!\!\!& \cdots, 
\end{array}
$$ 
which shows that $c_{20}<0$, and $c_{11} \ne 0$ if $n \ne \frac{1}{2}$. 
It should be pointed out that since we know $c_{20} \ne 0$ from the 
computation, we can assume the SNF in the form given in \eqref{Eqn36}. 
This implies that higher codimension B-T bifurcation can only occur when 
$c_{11} = 0$, leading to the so-called degenerate cusp B-T bifurcation. 
If $c_{20}$ can be zero, then we need to consider a different form of the 
SNF for the case $c_{20} = 0$.   

When $n = n_0 = \frac{1}{2}$, we have 
\begin{equation}\label{Eqn37} 
\kappa_0 = \frac{9}{16}, \quad e_0 = \frac{9}{2}, \quad 
x_1^0 = \frac{2}{9}, \quad y_1^0 = \frac{14}{9}, 
\end{equation}  
and 
$$ 
c_{20} =  - \dfrac{27}{64 }, \quad 
c_{11} = 0, \quad 
c_{31} =  - \dfrac{729}{1024} \ne 0.  \\[1.0ex] 
$$ 
Therefore, we have the following result. 

\begin{theorem}\label{Thm3.1} 
For system \eqref{Eqn11}, B-T bifurcation occurs from the 
endemic equilibrium ${\rm P_1}$: $ \big(\frac{n}{(n+1)^2}, 
\frac{n^2+n+1}{n(n+1)^2} \big)$ at the critical point 
$(e,\kappa) = \big( \frac{(n+1)^2}{1-n}, 
n(1-n)(n+1)^2 \big)$, with $0<n<1$. Moreover, 
the B-T bifurcation is 
\begin{enumerate}
\item[{\rm (i)}] 
codimension $2$ if $ n \in (0, \frac{1}{2})\bigcup (\frac{1}{2},1)$; or  
\item[{\rm (ii)}] 
codimension $3$ if $ n = \frac{1}{2}$.  
\end{enumerate} 
\end{theorem}

\subsection{Codimension-$2$ B-T bifurcation}

In \cite{LiZhouWuMa2007}, with $m$ involved in the parameter set,
a number of transformations were 
applied to obtain the normal form with unfolding up to second order. 
In the following, we apply our one-step transformation approach to 
derive the parametric simplest normal form (PSNF) 
\cite{GY2010,GY2012,GM2015,YuZhang2019}. Let 
\begin{equation}\label{Eqn38}
\kappa = n(1-n)(n+1)^2 + \mu_1, \quad  
e = \dfrac{(n+1)^2}{1-n} + \mu_2, 
\end{equation}
which, together with the transformation \eqref{Eqn34}, is substituted into  
\eqref{Eqn11} to yield the following system up to second-order terms,
\begin{equation}\label{Eqn39}
\begin{array}{rl}
\dfrac{{\rm d} u}{{\rm d} \tau}= \!\!\! & v, \\[2.0ex]
\dfrac{{\rm d} v}{{\rm d} \tau}= \!\!\! & \dfrac{1}{(1+n)^3(1-n)} \, \mu_1
+ \dfrac{n^2(1-n)}{(1+n)^3} \, \mu_2
+ \!\!\!\! \dss\sum_{i+j+k+l=2} \!\!\!\! p_{ijkl}\, u^i v^j \mu_1^k \mu_2^l, 
\end{array}
\end{equation}
where $p_{ijkl}$'s are coefficients expressed in terms of 
the parameter $n$. Next, applying the change of variables,
\begin{equation}\label{Eqn40}
\!\!\!
\begin{array}{rl}
u =\!\!\! & -\, \dfrac{1}{n^3(1+n)^3}\, y_1
      -\dfrac{1}{2 n^4 (1+n)^4}\, \beta_1
      + \dfrac{1}{n(1+n)^3}\, \beta_2
      - \dfrac{n^2-2}{4 n^7 (n+1)^5}\, \beta_1 y_1
\\[2.0ex]
& -\, \dfrac{ 2 n^2- n+1}{n^4(1+n)^4 (1-n)}\, \beta_2 y_1
 + \dfrac{ 1-n}{2n^6(1+n)^3}\, y_1^2
  - \dfrac{2n^3+n^2-5n-2}{6 n^7(1+n)^5}\, y_1 y_2,  
\\[2.5ex]
v =\!\!\! & -\, \dfrac{1}{n^3(1+n)^3}\, y_2
      -\dfrac{2n^3+n^2-5n-2}{6 n^7(1+n)^5}\, \beta_1 y_1
- \dfrac{n^2-2}{4 n^7(1+n)^5}\, \beta_1 y_2
\\[2.0ex]
& +\, \dfrac{2n^2-n+1}{n^4(1+n)^4(1-n)}\,\beta_2 y_2
+ \dfrac{1-n}{n^6(1+n)^3}\, y_1 y_2
  -\dfrac{2n^3+n^2-5n-2}{6 n^7(1+n)^5}\, y_2^2, 
\end{array}
\end{equation} 
and the parametrization,
\begin{equation}\label{Eqn41}
\begin{array}{rl}
\mu_1 =\!\!\! & -\, \dfrac{1-n}{n^3}\, \beta_1
- 2 n^2(1+n)\, \beta_2, 
\\[2.0ex]
\mu_2 =\!\!\! & \dfrac{2 (1+n)}{ (1-n)^2}\, \beta_2
+\dfrac{(1+n)(3n-1)}{n (1-n)^3}\, \beta_2^2, 
\end{array}
\end{equation} 
into \eqref{Eqn39}, we obtain the PSNF up to second-order terms:  
\begin{equation}\label{Eqn42}
\begin{array}{ll}
\dss\frac{{\rm d} y_1}{{\rm d} \tau} = &\!\!\! y_2, \\[2.0ex]
\dss\frac{{\rm d} y_2}{{\rm d} \tau} = &\!\!\!
\beta_1+\beta_2 y_2+y_1^2 - \dfrac{1-2n}{n^2}\, y_1 y_2, \qquad 
\textrm{for} \ \ 
n \in \Big(0,\, \dfrac{1}{2} \Big) \bigcup \Big(\dfrac{1}{2},\,1 \Big).  
\end{array}
\end{equation} 
Note from the above equation that the coefficient of 
$y_1 y_2$ is not normalized into $\pm 1$ in order to show the 
direct effect of the original system parameter $n$ on the dynamics of 
the system. It is clear that this coefficient can be positive or negative. 
Also, note that there is a negative multiplier $-\,\frac{1}{n^3(n+1)^3}$ 
in the transformation from $(u,v)$ to $(y_1,y_2)$. 

\vspace{0.05in} 
Based on the PSNF \eqref{Eqn42}, we have the following bifurcation result.

\begin{theorem}\label{Thm3.2}
For the system \eqref{Eqn11}, codimension-$2$ B-T bifurcation occurs
from the equilibrium ${\rm P_1}\!:$ $(x,y) \!=\!
(\frac{n}{(n+1)^2},\frac{n^2+n+1}{n(n+1)^2})$ 
when $\kappa \!=\! n(1-n)(n+1)^2$ and 
$ e \!=\! \frac{(n+1)^2}{1-n}$ if 
$n \! \in \! \big(0,\, \tfrac{1}{2} \big) \bigcup \big(\tfrac{1}{2},\,1 \big)$.
Moreover, three local bifurcations with the
representations of the bifurcation curves are given below.
\begin{enumerate}
\item[{\rm (1)}]
Saddle-node bifurcation occurs from the bifurcation curve:
$$
\hspace*{-1.28in} {\rm SN} = \left\{ (\beta_1,\beta_2) \mid \beta_1 =0,   
\left\{\!\! \begin{array}{ll}
\beta_2 <0 \ (0<n < \tfrac{1}{2}) \\[0.5ex]
\beta_2 >0 \ (\tfrac{1}{2}<n<1)
\end{array}
\right. \!\! 
\right\}.  
$$

\item[{\rm (2)}]
Hopf bifurcation occurs from the bifurcation curve:
$$
\hspace*{0.24in} {\rm H} =
\left\{ (\beta_1,\beta_2) \left| \beta_1
= -\, \frac{n^4}{(1\!-\!2n)^2} \, \beta_2^2, \right.  
\left\{\!\! \begin{array}{ll}
\beta_2 <0 \ (0<n < \tfrac{1}{2}), & \!\! {\rm supercritical}  \\[0.5ex]
\beta_2 >0 \ (\tfrac{1}{2}<n<1), & \!\! {\rm subcritical} 
\end{array}
\right. 
\!\!\! \right\}.  
$$

\item[{\rm (3)}]
Homoclinic loop bifurcation occurs from the bifurcation curve:
$$
\hspace*{0.25in} {\rm HL} = \left\{ (\beta_1,\beta_2) \left| \beta_1
=-\, \frac{49}{25} \frac{n^4}{(1\!-\!2n)^2}\, \beta_2^2, \right. 
\left\{\!\! \begin{array}{ll}
\beta_2 < 0 \ (0<n < \tfrac{1}{2}), & \!\! {\rm stable}  \\[0.5ex]
\beta_2 > 0 \ (\tfrac{1}{2}<n<1), & \!\! {\rm unstable} 
\end{array}
\right.
\!\!\! \right\}.  
$$
\end{enumerate}
\end{theorem}
The above formulas for bifurcation curves can be expressed in terms
of the original perturbation parameters 
$\mu_1$ and $\mu_2$ by using \eqref{Eqn40}.
The bifurcation diagram is depicted in Figure~\ref{fig4}.

\begin{figure}[!h]
\vspace*{-2.20in}
\begin{center}
\hspace*{-0.70in}
\begin{overpic}[width=1.19\textwidth,height=1.17\textheight]{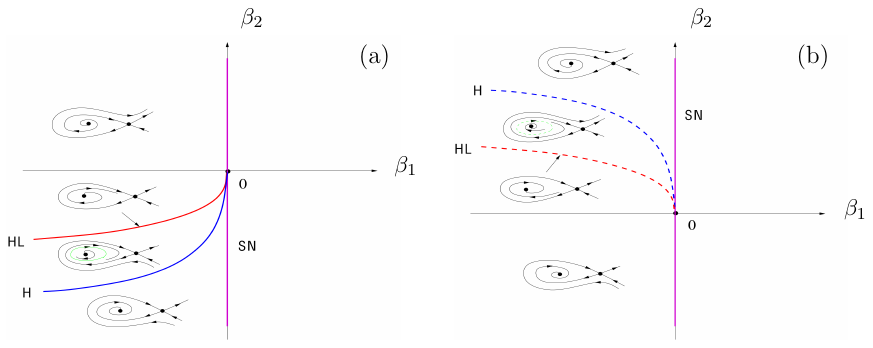}
\end{overpic}

\vspace*{-5.55in} 
\caption{Bifurcation diagrams for the codimension-2 B-T bifurcation
of the model \eqref{Eqn8} based on the normal form \eqref{Eqn42}: 
(a) for $0<n< \frac{1}{2}$; and (b) for $\frac{1}{2}<n<1$.}
\label{fig4}
\end{center} 
\vspace{-0.10in} 
\end{figure}

\begin{figure}[!t]
\vspace*{-2.25in}
\begin{center} 
\hspace*{-0.80in}
\begin{overpic}[width=1.19\textwidth,height=1.13\textheight]{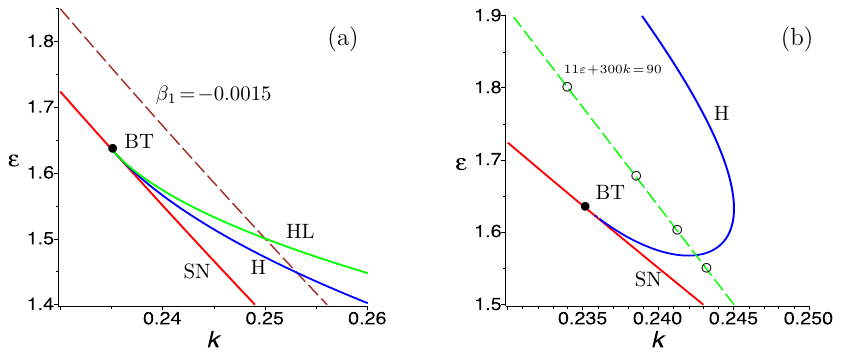}
\end{overpic}

\vspace{-5.25in}
\caption{Bifurcation diagrams for the codimension-$2$ B-T 
bifurcation of the model \eqref{Eqn8} with $m \!=\!2$ and 
$n \!=\! \frac{2}{5}$: (a) based on the PSNF \eqref{Eqn42}, 
where the straight line 
(in brown color) corresponds to the vertical line $\beta_1 \!=\! -0.0015$ 
in the bifurcation diagram in Figure~\ref{fig4}(a); and (b) 
based on the model \eqref{Eqn8}, with four points chosen from 
the line $ 11 \varepsilon \!+\! 300 k \!=\! 90$, with 
$\varepsilon \!=\! 1.55$, $1.60$, $1.676171875$ and $1.80$, respectively.}
\label{fig5} 
\end{center}

\vspace*{-0.30in}
\vspace*{-2.70in}
\begin{center} 
\hspace*{-0.83in}
\begin{overpic}[width=1.223\textwidth,height=1.276\textheight]{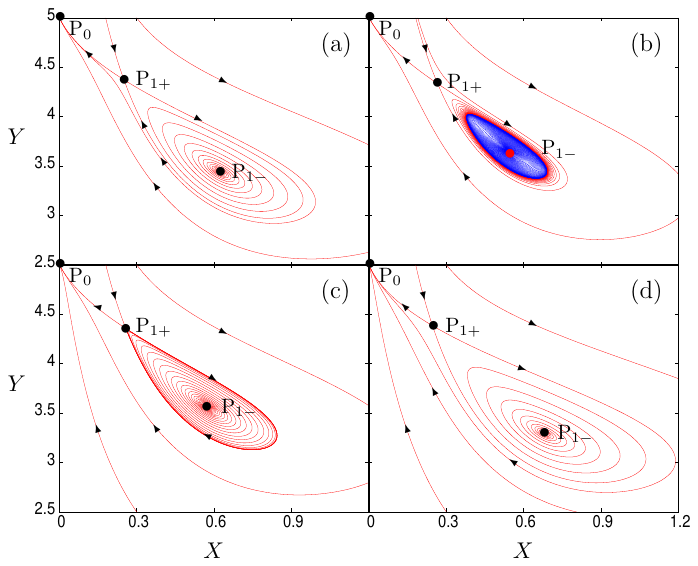}
\end{overpic}

\vspace{-4.40in}
\caption{Simulated phase portraits for the codimension-$2$ B-T 
bifurcation of the epidemic model \eqref{Eqn8} with $m \!=\!2$, 
$n \!=\! \frac{2}{5}$, and $k \!=\! \frac{90-11 \varepsilon}{300}$, for 
(a) $\varepsilon \!=\! 1.55$ showing the stable focus ${\rm P_{1-}}$, 
(b) $\varepsilon \!=\! 1.60$ showing Hopf bifurcation yielding 
a stalbe limit cycle and the unstable ${\rm P_{1-}}$, 
(c) $\varepsilon \!=\! 1.676171875$ showing the 
stable homoclinic loop and unstable ${\rm P_{1-}}$, and 
(d) $\varepsilon \!=\! 1.80$ showing the unstable focus ${\rm P_{1-}}$.} 
\label{fig6} 
\end{center}
\end{figure}

\begin{figure}[!t]
\vspace*{-2.20in}
\begin{center} 
\hspace*{-0.80in}
\begin{overpic}[width=1.17\textwidth,height=1.15\textheight]{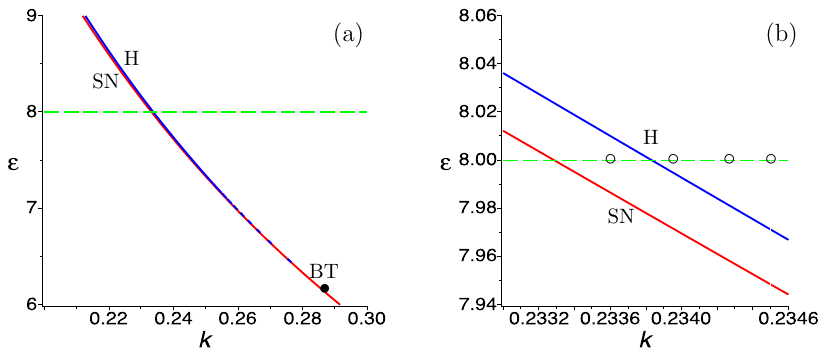}
\end{overpic}

\vspace{-5.35in} 
\caption{Bifurcation diagram for the codimension-$2$ B-T 
bifurcation of the epidemic model \eqref{Eqn8} with $m \!=\!2$, 
$n \!=\! \frac{3}{4}$: (a) a neighborhood of the B-T bifurcation 
point; and (b) the zoomed area along the line 
$\varepsilon \!=\! 8$, marked with four points by circles 
at $k \!=\! 0.2336$, $0.23395$, $0.23426542$ and $0.2345$.}
\label{fig7} 
\end{center} 

\vspace*{-0.30in}
\vspace*{-2.70in}
\begin{center} 
\hspace*{-0.83in}
\begin{overpic}[width=1.223\textwidth,height=1.276\textheight]{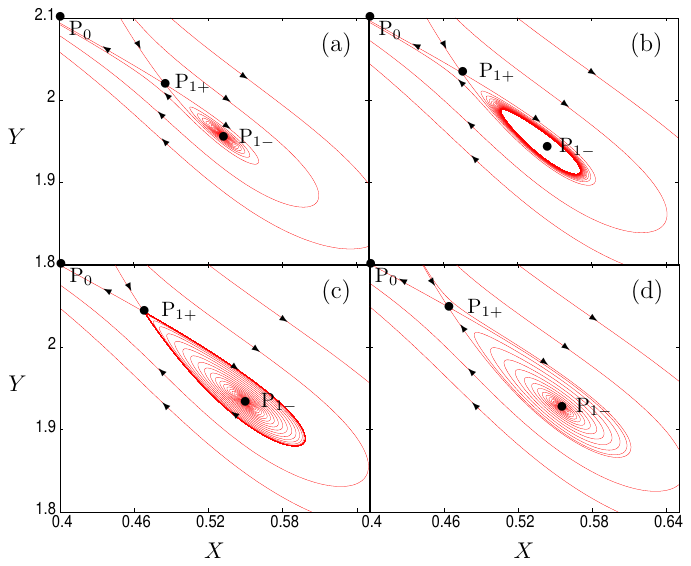}
\end{overpic}

\vspace{-4.40in}
\caption{Simulated phase portraits for the codimension-$2$ B-T 
bifurcation of the epidemic model \eqref{Eqn8} with $m \!=\!2$, 
$n \!=\! \frac{3}{4}$, and $\varepsilon \!=\! 8$, for 
(a) $k \!=\! 0.2336$ showing the unstable focus ${\rm P_{1-}}$, 
(b) $k \!=\! 0.23395$ showing Hopf bifurcation yielding 
an unstable limit cycle and the stable ${\rm P_{1-}}$, 
(c) $k \!=\! 0.23426542$ showing the 
unstable homoclinic loop and the stable ${\rm P_{1-}}$, and 
(d) $k \!=\! 0.2345$ showing the stable focus ${\rm P_{1-}}$.} 
\label{fig8} 
\end{center}
\end{figure}

In the following, we present simulations for the codimension-$2$ 
B-T bifurcations discussed above to illustrate 
the theoretical results. In order to give a direct impression of the 
original system's dynamical behaviours, we use the model \eqref{Eqn8}, 
rather than the normal forms \eqref{Eqn42} 
to perform the simulations. We choose $m\!=\! 2$, and two values of $n$: 
$ n \!=\! \frac{2}{5} \! \in \! (0, \frac{1}{2}) $, 
and $ n \!=\! \frac{3}{4} \! \in \! (\frac{1}{2},1) $. 
For $ n \!=\! \frac{2}{5}$, using the formulas given in 
\eqref{Eqn38}-\eqref{Eqn41}, we transform the bifurcation diagram 
in Figure~\ref{fig4}(a) back to that for the original model 
\eqref{Eqn8} near
the B-T critical point plotted in the $k$-$\varepsilon$ plane, as shown 
in Figure~\ref{fig5}(a). On the other hand, a bifurcation diagram 
directly based on the model \eqref{Eqn8} is given 
in Figure~\ref{fig5}(b). 
It is seen that the bifurcation diagram based on the 
normal form \eqref{Eqn42} (see Figure~\ref{fig4}(a)) 
gives a good indication for the dynamical 
behaviour of the original model near the B-T critical point. 
Moreover, the vertical line 
$\beta_1 \!=\! 0.01$ in the bifurcation diagram in Figure~\ref{fig4}(a) 
is mapped to a curve (almost a straight line) in Figure~\ref{fig5}(a). 
Therefore, we use the direct bifurcation diagram in Figure~\ref{fig5}(b) 
and choose four points on the straight line (in green color),  
$$ 
11 \varepsilon + 300 k = 90,  \quad \textrm{with} \quad 
\varepsilon = 1.55, \ 1.60, \ 1.676171875, \ 1.80. 
$$ 
The corresponding simulated phase portraits are shown in 
Figure~\ref{fig6}(a), (b), (c) and (d), respectively, 
representing the stable focus ${\rm P_{1-}}$, a stable limit cycle 
and the unstable ${\rm P_{1-}}$, the stable homoclinic loop, 
and the unstable ${\rm P_{1-}}$. These phase portraits exactly correspond 
to that in Figure~\ref{fig4}(a) in the upward direction.

Next, consider $n \!=\! \frac{3}{4}$. 
The bifurcation diagram is given in Figure~\ref{fig7}(a). 
The region between the 
saddle-node bifurcation curve and the Hopf bifurcation curve is 
quite narrow, and a zoomed region in shown in Figure~\ref{fig7}(b). 
Again, we choose four points from the line,
$$
\varepsilon = 8, \quad \textrm{with} \quad k=0.2336, \ 0.23395, \ 
0.23426542, \ 0.2345,
$$
and the corresponding 
simulated phase portraits are depicted in Figure~\ref{fig8}(a), (b), (c) 
and (d), respectively, representing the unstable focus ${\rm P_{1-}}$,
an unstable limit cycle and the stable ${\rm P_{1-}}$, the unstable 
homoclinic loop, and the stable ${\rm P_{1-}}$. These phase portraits 
exactly correspond to those in Figure~\ref{fig4}(b) in the downward direction.

\subsection{Codimension-$3$ B-T bifurcation}

In this section, we consider codimension-$3$ B-T bifurcation for 
the epidemic model \eqref{Eqn8}. First, we give a brief summary on the 
six-step transformation approach, and then introduce our one-step 
transformation method. 

\subsubsection{Summary of the six-step transformation method} 

In \cite{LiLiMa2015}, Li {\it et al.} applied the 
six-step transformation approach \cite{Dumortier1987} to 
provide a detailed analysis on the codimension-$3$ B-T bifurcation of 
the epidemic model \eqref{Eqn8}. In order to give a comparison, in the 
following, we first give a brief summary of the result 
(more details can be found in \cite{LiLiMa2015}). 
The basic idea of this approach is to employ a transformation in each 
step to remove one or two terms in the Taylor expansion of the vector 
field, which is not necessarily in algebraic formula and maybe in 
differential forms. 
Introducing the following transformation from the critical point,
$$
\quad k = \dfrac{9}{16m} + \mu_1, \quad 
\varepsilon = \dfrac{9}{2m} + \mu_2, \quad 
n=\dfrac{2m}{9} + \mu_3, \quad x = X- X_1^0, \quad y = Y - Y_1^0,   
$$
where $X_1^0=\frac{2m}{9}$, $Y_1^0=\frac{14m}{9}$, 
into \eqref{Eqn8} yields 
\begin{equation}\label{Eqn43}  
\frac{{\rm d} x}{{\rm d} \tau} 
= - \Big(\frac{1}{2} + \mu_3 \Big) x - y - \frac{14m}{9} 
\mu_3, \quad \frac{{\rm d} y}{{\rm d} \tau} = Q_1 (x,y), 
\end{equation}
where $Q_1$ is a Taylor expansion in $x,\,y$ and $\mu =(\mu_1,\mu_2,\mu_3)$. 
To simplify \eqref{Eqn43}, setting $u_1=x,\, v_1 = - (\frac{1}{2} + \mu_3 ) x 
- y - \frac{14m}{9} \mu_3$ yields 
$$
\frac{{\rm d}u_1}{{\rm d} \tau} = v_1, \quad \frac{{\rm d}v_1}{{\rm d} \tau} 
= Q_2 (u_1,v_1), 
$$
where 
$$ 
Q_2 (u_1,v_1) = \alpha_0 + \alpha_1 u_1 + \alpha_2 v_1 + \alpha_3 u_1 v_1 
+ a_1 u_1^2 + a_2 v_1^2 + a_3 u_1^3 + a_4 u_1^2 v_1 + a_5 u_1 v_1^2 
+ a_6 v_1^3, 
$$ 
in which $\alpha_j$'s and $a_k$'s are parameters, expressed in terms of 
the coefficients in $Q_1$. 

\vspace{0.05in} 
\noindent 
{\bf Step 1}. Use a nonlinear transformation $ u_1=u_2 + \frac{a_0}{2} u_2^2, 
\ v_1 = v_2 + a_2 u_2 v_2 $ to remove the $v_1^2$-term from $Q_2$, yielding 
$$
\frac{{\rm d} u_2}{{\rm d} \tau} = v_2, \quad 
\frac{{\rm d} v_2}{{\rm d} \tau} = Q_3 (u_2,v_2) 
+ R(u_2,v_2,\mu), 
$$ 
where the residue term $R$ has special property, and 
$$ 
Q_3 (u_2,v_2) = \beta_0 + \beta_1 u_2 + \beta_2 v_2 + \beta_3 u_2 v_2 
+ b_1 u_2^2 + b_2 u_2^3 + b_3 u_2^4 + b_4 u_2^2 v_2 + b_5 u_2 v_2^2 
+ b_6 v_2^3 +b_7 u_2^3 v_2, 
$$ 
where $\beta_j$'s and $b_k$'s are parameters, expressed in terms of 
the coefficients in $Q_2$. 

\vspace{0.05in} 
\noindent 
{\bf Step 2}. Use two nonlinear transformations, one in differential form 
$(u_2,v_2) \rightarrow (u_3,v_3)$ and one in algebraic form $(u_3,v_3) 
\rightarrow (u_4,v_4)$, to remove the $u_2 v_2^2$- and $v_2^3$-terms 
from $Q_3$, resulting in 
$$ 
\frac{{\rm d} u_4}{{\rm d} \tau} = v_4, \quad 
\frac{{\rm d} u_4}{{\rm d} \tau} = Q_5(u_4,v_4) + R(u_4,v_4,\mu). 
$$ 

\vspace{0.05in} 
\noindent 
{\bf Step 3}. Similarly, use a nonlinear transformation 
$(u_4,v_4) \rightarrow (u_5,v_5)$ to remove the $u_4^3$- and $u_4^4$-terms 
from $Q_5$, giving 
$$ 
\frac{{\rm d} u_5}{{\rm d} \tau} = v_5, \quad 
\frac{{\rm d} v_5}{{\rm d} \tau} = Q_6(u_5,v_5) + R(u_5,v_5,\mu). 
$$ 

\vspace{0.05in} 
\noindent 
{\bf Step 4}. Applying a nonlinear transformation  
$(u_5,v_5) \rightarrow (u_6,v_6)$ 
to remove the $u_5^2v_5$-term from $Q_6$ gives 
$$ 
\frac{{\rm d} u_6}{{\rm d} \tau} = v_6, \quad 
\frac{{\rm d} v_6}{{\rm d} \tau} = Q_7(u_6,v_6) + R(u_6,v_6,\mu), 
$$ 
where 
$$ 
Q_7(u_6,v_6) = \eta_0 + \eta_1 u_6 + \eta_2 v_6 + \eta_3 u_6 v_6  
+ \tilde{e}_1 u_6^2 + \tilde{e}_3 u_6^3 v_6. 
$$ 

\vspace{0.05in} 
\noindent 
{\bf Step 5}. Normalizing $\tilde{e}_1$ and $\tilde{e}_3$ 
to be $1$ in $Q_7$ yields 
$$ 
\frac{{\rm d} u_6}{{\rm d} \tau} = v_6, \quad 
\frac{{\rm d} v_6}{{\rm d} \tau} = Q_8(u_6,v_6) + R(u_6,v_6,\mu), 
$$ 
where 
$$ 
Q_8(u_6,v_6) = \sigma_0 + \sigma_1 u_6 + \sigma_2 v_6 + \sigma_3 u_6 v_6  
+ u_6^2 + u_6^3 v_6. 
$$ 

\vspace{0.05in} 
\noindent 
{\bf Step 6}. Use a nonlinear transformation 
$(u_6,v_6) \rightarrow (y_1,y_2)$ to 
remove the $u_6$-term from $Q_8$, finally yielding the normal form,  
\begin{equation}\label{Eqn44} 
\begin{array}{ll}  
\dfrac{{\rm d} y_1}{{\rm d} \tau} = y_2, \\[2.0ex] 
\dfrac{{\rm d} y_2}{{\rm d} \tau} = \varepsilon_1 + \varepsilon_2 y_2 
+ \varepsilon_3 y_1 y_2 + y_1^2 + y_1^3 y_2 + R(y_1,y_2,\mu). 
\end{array}
\end{equation}  
A careful checking on the transformations given in \cite{LiLiMa2015} we 
obtain 
$$ 
\tilde{e}_1 = \dfrac{27}{64 m}, \quad \tilde{e}_3 = -\, \dfrac{3645}{1024 m^3}. 
$$ 
Moreover, combining the transformations from 
$(\mu_1,\mu_2,\mu_3)$ to $(\varepsilon_1,\varepsilon_2,\varepsilon_3)$ yields 
\begin{equation}\label{Eqn45}  
\det \left[\frac{\partial (\mu_1,\mu_2,\mu_3)} 
     {\partial (\varepsilon_1,\varepsilon_2,\varepsilon_3)} \right]_{\mu=0} 
= -\, \frac{6}{m^2}\, \tilde{e}_1^{\frac{12}{5}} \tilde{e}_3^{-\frac{4}{5}} 
= -\, \frac{27}{64 m^2} \, \frac{(72)^{\frac{1}{5}}}{5^{\frac{4}{5}}} 
\ne 0, 
\end{equation}
implying that system \eqref{Eqn44} with 
$(\varepsilon_1,\varepsilon_2,\varepsilon_3) \sim (0,0,0)$ for 
$(y_1,y_2)$ near $(0,0)$ is equivalent to system \eqref{Eqn8} with 
$(k,\varepsilon,n) \sim (k_0,\varepsilon_0,n_0)$ for 
$(X,Y)$ near $(X_1^0,Y_1^0)$. 
 
It should be pointed out that although the computation demand in each step 
of the above process is not heavy, finding the complete transformation 
between the system \eqref{Eqn8} and the final normal 
form \eqref{Eqn44} is not an easy task. 
Also note that in each of the above transformations, only the 
dominant terms $Q_k$ are take into account, while 
the $R$'s with special property are ignored, which reduces 
computation demanding.

\subsubsection{One-step transformation method} 

Now, we turn to consider our one-step transformation approach, 
which is based on the parametric simplest normal form (PSNF)
\cite{YL2003,GM2015,YuZhang2019}. The main difficulty of this method is 
how to determine the basis for nonlinear transformations, since different 
systems require different forms of transformations. 

Introducing the transformation,
$$ 
x = x_1^0 + u, \quad y= y_1^0 + v, \quad 
\kappa = \kappa_0 + \mu_1, \quad e = e_0 + \mu_2, 
\quad n = n_0 + \mu_3, 
$$
together with \eqref{Eqn37} into \eqref{Eqn11}    
yields the following system up to 4th order,
\begin{equation}\label{Eqn46} 
\begin{array}{rl}
\dfrac{{\rm d} u}{{\rm d} \tau}= \!\!\! & v, \\[1.0ex]
\dfrac{{\rm d} v}{{\rm d} \tau}= \!\!\! & \dfrac{16}{27} \, \mu_1
+\dfrac{1}{27} \, \mu_2
+ \!\!\! \dss\sum_{i+j+k+l+s=2}^4 \!\!\!\! p_{ijkls}\ 
u^i v^j \mu_1^k \mu_2^l \mu_3^s + \textrm{h.o.t.}, 
\end{array}
\end{equation} 
where the coefficients $p_{ijkls}$ are real values. 
Then, applying the change of variables, 
$$
\begin{array}{rl} 
u =\!\!\!& - \tfrac{4}{9} 48^{\frac{1}{5}}  y_1 
- \tfrac{1}{6} 48^{\frac{2}{5}} \beta_1 
+ \tfrac{1}{27} 48^{\frac{3}{5}} \beta_2  
+ \tfrac{1}{18} 48^{\frac{2}{5}} y_1^2 
+ \tfrac{1}{9} 48^{\frac{2}{5}} y_1 y_2  
+ \tfrac{1}{240} 48^{\frac{3}{5}} y_2^2 
\\[1.5ex] 
& \quad +\, \big( \tfrac{763}{2160} 48^{\frac{3}{5}} \beta_1 
           - \tfrac{8}{81} 48^{\frac{4}{5}} \beta_2 
           + \tfrac{16}{243} 48^{\frac{3}{5}} \beta_3  \big)\, y_1  
\\[1.5ex] 
& \quad +\, \big( \tfrac{2386838}{93555} 48^{\frac{1}{5}} \beta_1 
           - \tfrac{24223}{93555} 48^{\frac{2}{5}} \beta_2 
           + \tfrac{8306}{5103} 48^{\frac{1}{5}} \beta_3  \big) \, y_2 
+ \tfrac{103439}{259200} 48^{\frac{4}{5}} \beta_1^2  
+ \tfrac{22}{81} 48^{\frac{1}{5}} \beta_2^2  
\\[0.5ex] 
& \quad -\, \tfrac{32}{9} \beta_1 \beta_2  
+ \tfrac{5}{162} 48^{\frac{4}{5}} \beta_1 \beta_3  
- \tfrac{32}{243} \beta_2 \beta_3  
+ \!\!\!\!\! \dss\sum_{i+j+k+l+s=3}^4 \!\!\!\!\! a_{ijkls}\ y_1^i y_2^j
\beta_1^k \beta_2^l \beta_3^s ,  
\\[2.0ex] 
v =\!\!\!& - \tfrac{ 48^{\frac{4}{5}} }{36} y_2 
+ \tfrac{7}{54} 48^{\frac{3}{5}} y_2^2 
+ \tfrac{7}{54} 48^{\frac{3}{5}} \beta_1 \, y_1 
+ \big( \tfrac{149}{144} 48^{\frac{1}{5}} \beta_1 
           - \tfrac{2}{9} 48^{\frac{2}{5}} \beta_2 
           + \tfrac{2}{27} 48^{\frac{1}{5}} \beta_3  \big) \, y_2 
\\[0.5ex] 
& \quad 
+\, \tfrac{1193419}{748440} 48^{\frac{4}{5}} \beta_1^2  
+ \tfrac{24233}{31185} \beta_1 \beta_2  
+ \tfrac{4153}{40824} 48^{\frac{4}{5}} \beta_1 \beta_3  
+ \!\!\!\!\! \dss\sum_{i+j+k+l+s=3}^4 \!\!\!\!\! b_{ijkls}\ y_1^i y_2^j
\beta_1^k \beta_2^l \beta_3^s ,  
\end{array} 
$$
where the coefficients $b_{ijkls}$ are real values, 
the parametrization, 
$$
\begin{array}{rl} 
\mu_1 = \!\!\! & \tfrac{9}{64} 48^{\frac{2}{5}} \beta_1 
+ 3 108^{\frac{1}{5}} \beta_2 
+ \tfrac{15}{128} 48^{\frac{4}{5}} \beta_1^2  
+ \tfrac{41}{64} 48^{\frac{1}{5}} \beta_2^2  
- \tfrac{59}{16} \beta_1 \beta_2 
\\[0.5ex] 
& \quad +\, \tfrac{11}{384} 48^{\frac{4}{5}} \beta_1 \beta_3  
- \tfrac{31}{48} \beta_2 \beta_3  
+ \!\!\!\!\! \dss\sum_{i+j+k=3}^4 \!\!\!\!\! \alpha_{ijk}\ \beta_1^i 
\beta_2^j \beta_3^k ,  
\\[1.5ex] 
\mu_2 = \!\!\! & - 9 72^{\tfrac{1}{5}}\beta_1
       +3 108^{\tfrac{1}{5}} \beta_2
       -3 162^{\tfrac{1}{5}}\beta_1^2
       -8 48^{\tfrac{1}{5}}\beta_2^2
       +38\beta_1\beta_2 
\\[1.0ex] 
& \quad -\, \tfrac{13}{3} 162^{\tfrac{1}{5}} \beta_1\beta_3
       +\tfrac{31}{3} \beta_2\beta_3
       +\tfrac{34}{3} 162^{\tfrac{1}{5}} \beta_2^3
       -\tfrac{115}{9} 108^{\tfrac{1}{5}} \beta_2^2\beta_3
       +\tfrac{25}{6} 72^{\tfrac{1}{5}} \beta_2\beta_3^2 , 
\\[1.0ex] 
\mu_3 = \!\!\! & \tfrac{5}{12} 72^{\tfrac{1}{5}} \beta_1
       -\tfrac{1}{4} 108^{\tfrac{1}{5}} \beta_2
       -\tfrac{1645571}{51840} 162^{\tfrac{1}{5}} \beta_1^2
       +\tfrac{1}{6} 48^{\tfrac{1}{5}} \beta_2^2
       +\tfrac{2}{81} 162^{\tfrac{1}{5}} \beta_3^2
\\[1.0ex] 
&-\, \tfrac{9119}{1800} \beta_1 \beta_2
       -\tfrac{31}{810} 162^{\tfrac{1}{5}} \beta_1 \beta_3
       -\tfrac{1}{3} \beta_2 \beta_3,
\end{array} 
$$
and the time rescaling, 
$$
{\rm d} \tau = \Big( \tfrac{1}{4} 108^{\tfrac{1}{5}} 
       -\tfrac{1}{4} 162^{\tfrac{1}{5}} y_1
       -\tfrac{1}{6} 48^{\tfrac{1}{5}} \beta_2
       +\tfrac{5}{18} \beta_3 
\Big) \, {\rm d} \tau_1,  
$$
into \eqref{Eqn46} yields the following PSNF up to $4$th-order terms,
\begin{equation}\label{Eqn47}
\begin{array}{rl}
\displaystyle\frac{{\rm d} y_1}{{\rm d} \tau_1} = &\!\!\! y_2, \\[1.5ex]
\displaystyle\frac{{\rm d} y_2}{{\rm d} \tau_1} = &\!\!\! \beta_1
+ \beta_2 \, y_2 + \beta_3 \, y_1 y_2 + y_1^2 + y_1^3 y_2
+ \mathcal{O}(|(y_1,y_2,\beta)|^5).
\end{array}
\end{equation}

It is easy to verify that
\begin{equation}\label{Eqn48}
\det \! \left[ \frac{\partial (\mu_1, \mu_2, \mu_3)}
{\partial (\beta_1, \beta_2, \beta_3)} \right]_{\beta=0}
= -\, \dfrac{27}{64}\, 72^{\frac{1}{5}} \ne 0, 
\end{equation}
which shows that near the critical point $\mu=0$, 
system \eqref{Eqn8} has the same bifurcation set with respect to
$\mu $ as system \eqref{Eqn47} has that with respect to $\beta$,
up to a homeomorphism in the parameter space.

Comparing our one-step transformation method with the classical 
six-step transformation approach, we have the following observations. 
\begin{enumerate}
\item[{(i)}]  
The one-step approach depends upon purely algebraic computation; while
the six-step approach involves different types of transformations.

\item[{(ii)}]  
The one-step approach yields a direct relation between the SNF (or PSNF)
and the original system, which makes it convenient in applications; while for
the six-step approach, finding a direct relation needs to put all the
transformation together, which involves a lot of computations.

\item[{(iii)}]  
The one-step approach is easier to be used for developing a general
algorithm for the symbolic computation.

\item[{(iv)}] 
The one-step approach provides a complete nonlinear transformation 
up to a given order, while the six-step approach only take the 
dominant parts in each step of transformation.  

\item[{(v)}]  
The six-step approach has less computation in each step; while the
computation demands for the one-step is higher, in particular for
higher-codimension bifurcations.
\end{enumerate}  

Now, following the method described in \cite{Dumortier1987}, and the 
computations in \cite{YuZhang2019}, we apply 
the method of normal forms and Abelian integral 
(or the Melnikov function method) to derive the bifurcations 
for the codimension-$3$ B-T bifurcation. 
In \cite{YuZhang2019}, the term $x_1^3 x_2$ (same as 
$y_1^3 y_2$ in \eqref{Eqn47}) in the normal form 
has a coefficient $-\,b_1$. So,  
theoretically speaking, we can use the formulas in \cite{YuZhang2019} and 
set $ b_1 \!=\! -1$ to directly obtain the bifurcation results for 
our system \eqref{Eqn47}. However, 
it has been noted that there are some errors in the formulas 
given in \cite{YuZhang2019} for the codimension-3 B-T bifurcation, 
as listed below. 
\begin{enumerate}
\item[{\rm (a)}]
In Eqn. (91), $\xi_3 + 3 b_1 \xi_1$ should be 
$\xi_3 - 3 b_1 \xi_1$. 

\item[{\rm (b)}]
In Eqn. (93), $\frac{103}{55}$ 
should be $\frac{179}{11}$. 

\item[{\rm (c)}]
In Eqn. (107), $z_1(t) = - 3 \,{\rm sech}^2(t)$ 
and $ z_2(t) = 3 \, {\rm sech}^2(t) \, {\rm tanh} (t)$ should be 
$z_1(t) = - 3 \,\bar{\nu}_1 \, {\rm sech}^2(t)$ 
and $ z_2(t) = 3 \, \bar{\nu}_1\, {\rm sech}^2(t) \, 
 \, {\rm tanh} (t)$, respectively, and thus $-\frac{103}{77}$ 
should be $\frac{895}{77}$. 

\item[{\rm (d)}]
In Eqn. (109), 
$\frac{103}{55}$ should be $\frac{179}{11}$. 
\end{enumerate}
Then, other changes in Eqns. (112)-(115) are followed accordingly. 
However, note that the bifurcation results shown in Figure 19 
of \cite{YuZhang2019} are qualitatively not changed.  

In order to correct the errors in \cite{YuZhang2019} and provide 
readability for readers,
in the following we briefly describe the derivations. 
First of all, it is easy to see that system \eqref{Eqn47} has two 
equilibrium solutions ${\rm E_{\pm}}$, 
\begin{equation}\label{Eqn49}
{\rm \widetilde{E}_\pm}= (y_{1 \pm},0), \quad {\rm where} \quad
y_{1 \pm} = \pm \sqrt{-\beta_1} \quad {\rm for} \quad \beta_1 < 0.
\end{equation} 
The Jacobian of \eqref{Eqn47} evaluated at ${\rm \widetilde{E}_\pm}$ is given by
$$
J_\pm = \left[ \begin{array}{ll} 0 & 1 \\
2 y_{1 \pm} & \beta_2 + \beta_3 y_{1 \pm}
+ y_{1 \pm}^3 \end{array}
\right],
$$
which indicates that 
$ {\rm \widetilde{E}_{+}}$ is a saddle, and ${\rm \widetilde{E}_{-}} $ 
is either a focus or node.
It is easy to see from the Jacobian that the plane 
\begin{equation}\label{Eqn50}
{\rm SN} = \big\{ (\beta_1, \beta_2, \beta_3) \mid \beta_1 = 0 \big\},
\end{equation}
excluding the origin in the
parameter space is the saddle-node bifurcation surface. Hopf bifurcation 
occurs from ${\rm \widetilde{E}_-}$ on the critical surface, defined by 
that the trace equals zero, i.e.,  
$$
\beta_2 - (\beta_3 - \beta_1) \sqrt{-\beta_1} = 0, \quad 
(\beta_1<0). 
$$
Based on Hopf bifurcation theory, a direct computation 
(e.g., with the Maple program in \cite{Yu1998}) yields the 
following focus values, 
$$
v_1 = \dfrac{\beta_3 + 3 \beta_1}{16 \sqrt{-\beta_1}} \quad 
\textrm{and} \quad \left. v_2 \right|_{v_1 = 0} 
= \dfrac{5}{96  \sqrt{-\beta_1}} > 0,
$$ 
which implies that generalized Hopf bifurcation occurs on the surface, 
defined by $v_1=0$, 
\begin{equation}\label{Eqn51}
\beta_3 + 3 \beta_1 = 0, \quad (\beta_1<0), 
\end{equation}
leading to two limit cycles, with 
the outer one unstable and the inner one stable, and 
both of them enclose the unstable focus ${\rm \widetilde{E}_{-}}$.  

Next, to find the homoclinic and the degenerate 
homoclinic bifurcations, we apply the Melnikov function method 
\cite{HanYu2012}. Introducing the scaling,
$$
y_1 = \varepsilon^{\frac{2}{5}} w_1, \ \
y_2 = \varepsilon^{\frac{3}{5}} w_2, \ \
\beta_1 = \varepsilon^{\frac{4}{5}} \varphi_1, \ \
\beta_2 = \varepsilon^{\frac{6}{5}} \varphi_2, \ \
\beta_3 = \varepsilon^{\frac{4}{5}} \varphi_3, \ \
\tau_2 = \varepsilon^{\frac{1}{5}} \tau_1, \ \
(0< \varepsilon \ll 1),
$$
together with the following transformation, 
$$
w_1 = \bar{\varphi}_1 + \tilde{x}_1, \quad
w_2 = \sqrt{2 \bar{\varphi}_1} \, \tilde{x}_2, \quad
\tau_3 = \sqrt{2 \bar{\varphi}_1} \, \tau_2, \quad
\varphi_1 = - \bar{\varphi}_1^2, \ \ (\bar{\varphi}_1>0),
$$
into \eqref{Eqn47} we obtain
\begin{equation}\label{Eqn52}
\begin{array}{rl} 
\dss\frac{{\rm d} \tilde{x}_1}{{\rm d} \tau_3} = \!\!\! & \tilde{x}_2,\\[2.0ex]
\dss\frac{{\rm d} \tilde{x}_2}{{\rm d} \tau_3} = \!\!\! &
\tilde{x}_1 + \dfrac{1}{2 \bar{\varphi}_1}\, \tilde{x}_1^2 +
\varepsilon \, q(\tilde{x}_1,\tilde{x}_2, \bar{\varphi}),
\end{array}
\end{equation}
where 
$$
q(\tilde{x}_1,\tilde{x}_2, \bar{\varphi}) 
= \dfrac{1}{\sqrt{2 \bar{\varphi}_1}} \,
\big[ (\varphi_2 + \bar{\varphi}_1 \, \varphi_3 + \bar{\varphi}_1^3 ) 
\tilde{x}_2 + ( \varphi_3 + 3 \bar{\varphi}_1^2 )  \tilde{x}_1 \tilde{x}_2
+ 3 \bar{\varphi}_1 \tilde{x}_1^2 \tilde{x}_2 
+ \tilde{x}_1^3 \tilde{x}_2 \big], 
$$
with $\bar{\varphi}=(\bar{\varphi}_1,\varphi_2,\varphi_3)$.

The system $\left. \eqref{Eqn52}\right|_{\varepsilon=0}$ is a Hamiltonian
system with two equilibrium solutions,
$$
{\rm \bar{E}_-} = ( -2 \bar{\varphi}_1,0) \quad {\rm and} \quad
{\rm \bar{E}_0} = (0, 0),
$$
with ${\rm \bar{E}_- }$ and ${\rm \bar{E}_0}$
being center and saddle, respectively.
These two equilibria correspond to the ${\rm \widetilde{E}_\pm}$ defined 
in \eqref{Eqn49}. The Hamiltonian is given by 
$$
H(\tilde{x}_1,\tilde{x}_2) = \dfrac{1}{2}\, (\tilde{x}_2^2 - \tilde{x}_1^2)
- \dfrac{1}{6 \bar{\varphi}_1} \, \tilde{x}_1^3,
$$
and the homoclinic orbit connecting ${\rm E_0}$ is described by
$$
\Gamma_0: \quad H (\tilde{x}_1,\tilde{x}_2)
= \dfrac{1}{2}\, (\tilde{x}_2^2 - \tilde{x}_1^2)
- \dfrac{1}{6 \bar{\varphi}_1} \, \tilde{x}_1^3,
\quad  \textrm{with} \ \ H(0,0) = 0,
$$
and $H(-2 \bar{\varphi}_1,0) = - \frac{2}{3} \, \bar{\varphi}_1^2$.
Thus, any closed orbits of the Hamiltonian system
$\left. \eqref{Eqn52}\right|_{\varepsilon=0}$ inside
$\Gamma_0$ can be described by
$$
\Gamma_{\rm h}: \quad H(\tilde{x}_1,\tilde{x}_2,h) =
 \dfrac{1}{2}\, (\tilde{x}_2^2 - \tilde{x}_1^2)
- \dfrac{1}{6 \bar{\varphi}_1} \, \tilde{x}_1^3 - h = 0, \quad
h \in \Big(\!\! -\! \frac{2}{3} \, \bar{\varphi}_1^2, 0 \Big).
$$
Now, the Abelian integral or the (first-order) Melnikov function
for the perturbed system \eqref{Eqn52} can be written as \cite{HanYu2012} 
$$
\begin{array}{rl} 
M(h,\varphi) \!\!\! & = \dss\oint_{\Gamma_{\rm h}} 
\big[q (\tilde{x}_1,\tilde{x}_2,\varphi)
\, {\rm d} \tilde{x}_1 - p(\tilde{x}_1,\tilde{x}_2,\varphi) \, {\rm d} 
\tilde{x}_2 \big]_{\varepsilon=0}  \qquad (p=0) \\[2.5ex]
& = \dss\oint_{\Gamma_{\rm h}}  q (\tilde{x}_1,\tilde{x}_2,\varphi) 
\! \mid_{\varepsilon=0} \, {\rm d} \tilde{x}_1 
= \dss\oint_{\Gamma_{\rm h}} H_{\tilde{x}_2}
 q (\tilde{x}_1,\tilde{x}_2,\varphi) \! \mid_{\varepsilon=0} \, 
{\rm d} \tau_3 \\[2.5ex]
& = \dfrac{1}{\sqrt{2 \bar{\varphi}_1}}  \dss\oint_{\Gamma_{\rm h}}
\tilde{x}_2^2\, \big[ \varphi_2 + \bar{\varphi}_1 \, \varphi_3 
+ \bar{\varphi}_1^3  + ( \varphi_3 + 3 \bar{\varphi}_1^2 ) \tilde{x}_1
+ 3 \bar{\varphi}_1 \tilde{x}_1^2 + \tilde{x}_1^3  \big] \, {\rm d} \tau_3 
\\[2.5ex] 
& = C_0(\varphi) + C_1(\varphi)\, h \ln |h|
+ C_2(\varphi) \, h + C_3(h) \, h^2 \ln |h| + \cdots, 
\end{array} 
$$
for $0< -h \ll 1$, where 
$$
\begin{array}{ll}
C_0(\varphi) = \dss\frac{1}{\sqrt{2 \bar{\varphi}_1}}  \dss\oint_{\Gamma_0}
\tilde{x}_2^2 \big[ \varphi_2 + \bar{\varphi}_1 \, \varphi_3 
+ \bar{\varphi}_1^3 + ( \varphi_3 + 3 \bar{\varphi}_1^2 )  \tilde{x}_1
+ 3 \bar{\varphi}_1 \tilde{x}_1^2 + \tilde{x}_1^3 \big] \, {\rm d} \tau_3, 
\\[2.0ex]
C_1(\varphi) = a_{10} + b_{01},
\end{array}
$$
in which $a_{10} $ and $b_{01}$ are the coefficients in the functions
$p(\tilde{x}_1,\tilde{x}_2,\varphi)$ and 
$q(\tilde{x}_1,\tilde{x}_2,\varphi)$, given by
$$
a_{10} =0, \quad
b_{01} = \dss\frac{1}{\sqrt{2 \bar{\varphi}_1}}
(\varphi_2 + \bar{\varphi}_1 \, \varphi_3 + \bar{\varphi}_1^3 ).
$$
To compute $C_0(\varphi)$, introducing the parametric transformation,
$$ 
\tilde{x}_1(\tau_3) = -3 \,\bar{\varphi}_1 \, {\rm sech}^2(\tau_3), \quad
\tilde{x}_2(\tau_3) = 3 \, \bar{\varphi}_1\, {\rm sech}^2(\tau_3) \, 
{\rm tanh} (\tau_3),
$$
into $C_0(\varphi)$ with a direct integration we obtain
$$
C_0(\varphi) = \frac{6 \bar{\varphi}_1 \sqrt{2 \bar{\varphi}_1}}{5}
\Big( \varphi_2 - \frac{5}{7} \bar{\varphi}_1 \, \varphi_3 
- \frac{895}{77} \, b_1 \bar{\varphi}_1^3 \Big).  
$$
Finally, we express $C_0(\varphi)$ and $C_1(\varphi)$
in terms of the original perturbation parameters $\beta_j$ by using
$$
\bar{\varphi}_1 = \sqrt{-\varphi_1} 
= \sqrt{- \varepsilon^{-\frac{4}{5}} \, \beta_1}
=  \varepsilon^{-\frac{2}{5}}  \sqrt{-\beta_1}, \quad
\varphi_2 =  \varepsilon^{-\frac{6}{5}} \beta_2, \quad
\varphi_3 =  \varepsilon^{-\frac{4}{5}} \beta_3,
$$
as
\begin{equation}\label{Eqn53}
\begin{array}{ll} 
C_0(\beta) = \dfrac{6 \bar{\varphi}_1 \sqrt{2 \bar{\varphi}_1}}{5}
\, \varepsilon^{-\frac{6}{5}}
\Big[ \beta_2 - \dfrac{5}{7} \Big( \beta_3 - \dfrac{179}{11}\,\beta_1 \Big)
\sqrt{- \beta_1} \Big], \\[1.5ex]
C_1(\beta) = \dfrac{1}{ \sqrt{2 \bar{\varphi}_1}}
\, \varepsilon^{-\frac{6}{5}}
\left[ \beta_2 + ( \beta_3 - \beta_1 ) \sqrt{- \beta_1} \, \right].
\end{array}
\end{equation}
Hence, the homoclinic and degenerate homoclinic bifurcation surfaces 
are defined by $ C_0(\beta) \!=\! 0$ and $C_1(\beta)=0$, respectively. 

Summarizing the above results we obtain the following theorem. 
Note that a summary of the result was given 
in \cite{LiLiMa2015}. Here, more detailed formulas are provided.

\begin{theorem}\label{Thm3.3}
For the epideminc model \eqref{Eqn11}, codimension-$3$ B-T bifurcation occurs
from the equilibrium ${\rm P_1}$: $(x,y) \!=\!
(\frac{2}{9},\frac{14}{9})$ when $ \kappa \!=\! \frac{9}{16}$, 
$e \!=\! \frac{9}{2}$ and $n \! =\! \frac{1}{2}$.
Moreover, six local bifurcations with the
representations of the bifurcation surfaces/curves are obtained, 
as given below.
\begin{enumerate}
\item[{\rm (1)}]
Saddle-node bifurcation occurs from the critical surface: 
$$
\hspace*{-2.65in} 
{\rm SN} = \big\{ (\beta_1,\beta_2,\beta_3) \mid \beta_1=0 \big\}.
$$

\item[{\rm (2)}]
Hopf bifurcation occurs from the critical surface: 
$$
\hspace*{-1.24in} {\rm H} = \big\{ (\beta_1,\beta_2,\beta_3) \mid
\beta_1<0, \ \beta_2 =
\big( \beta_3 - \beta_1 \big) \sqrt{- \beta_1 } \big\}.
$$

\item[{\rm (3)}]
Homoclinic loop bifurcation occurs from the critical surface: 
$$
\hspace*{-0.87in} {\rm HL} =
\left\{ (\beta_1,\beta_2,\beta_3) \mid \beta_1<0, \
\beta_2 = \tss\frac{5}{7} \big( \beta_3
- \tss\frac{179}{11} \beta_1 \big) \sqrt{-\beta_1} \right \}.
$$

\item[{\rm (4)}]
Generalized Hopf bifurcation occurs from the critical curve: 
$$
\hspace*{-0.60in} {\rm GH} = 
\left\{ (\beta_1,\beta_2,\beta_3) \mid \beta_1<0, \
\beta_2 = -4 \beta_1 \sqrt{-\beta_1}, \ 
\beta_3 = -3 \beta_1 \right \}.
$$

\item[{\rm (5)}]
Degenerate homoclinic bifurcation occurs from the critical curve: 
$$
\hspace*{-0.45in} {\rm DHL} = \left\{ (\beta_1,\beta_2, \beta_3) 
\left| \, \beta_1<0, \ 
\beta_2 = -\tss\frac{70}{11} \, \beta_1 \sqrt{-\beta_1} , \
\beta_3 = \tss\frac{81}{11} \, \beta_1 
\right.
\right\}. 
$$

\vspace{0.05in} 
\item[{\rm (6)}]
Double limit cycle bifurcation occurs from a critical surface,
which is tangent to the Hopf bifurcation surface ${\rm H}$ on
the critical curve ${\rm GH}$, and tangent to the homoclinic bifurcation
surface ${\rm HL}$ on the critical curve ${\rm DHL}$.
\end{enumerate}
\end{theorem}

\begin{figure}[!h]
\vspace*{-0.60in}
\begin{center} 
\hspace*{-0.90in}
\begin{overpic}[width=1.26\textwidth,height=1.143\textheight]{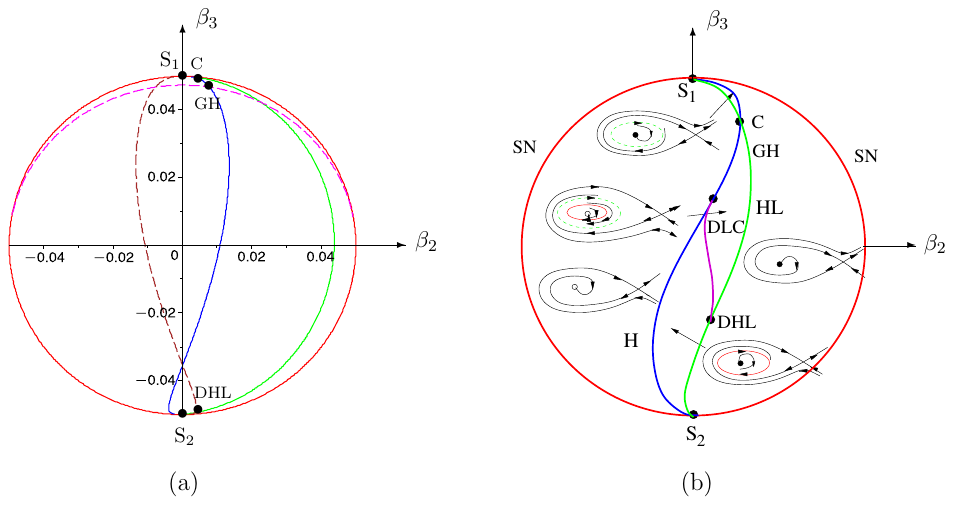}
\end{overpic}

\vspace{-6.20in}
\caption{Bifurcation diagram for the codimension-$3$ B-T bifurcation
based on the normal form \eqref{Eqn47}, 
displayed in the intersection of the cone and the
$2$-sphere $\beta_1^2 + \beta_2^2 + \beta_3^2 = \sigma^2$,
with the red color curve for saddle-node, blue curve for Hopf and green
curve for homoclinic loop bifurcations, respectively: 
(a) with $\sigma \!=\! 0.05$, where the intersection point of the pink and 
blue curves is the degenerate Hopf bifurcation, and the intersection point
of the brown and green curves denotes the degenerate homoclinic
loop bifurcation; and (b) a schematic bifurcation diagram, 
where the GH and DHL represent the generalized Hopf critical point and 
the degenerate homoclinic critical point, respectively.}
\label{fig9} 
\end{center}

\vspace*{-0.30in}
\vspace*{-2.20in}
\begin{center} 
\hspace*{-0.85in}
\begin{overpic}[width=1.250\textwidth,height=1.157\textheight]{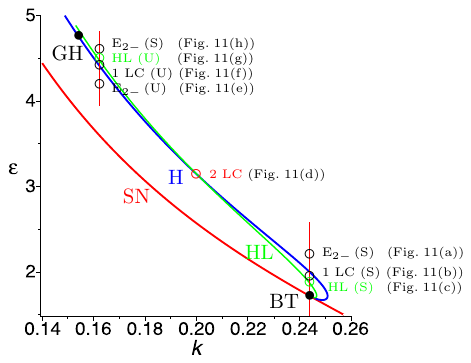}
\end{overpic}

\vspace{-5.35in}

\caption{Bifurcation diagram for the codimension-$3$ B-T 
bifurcation of the epidemic model \eqref{Eqn8} with $m \!=\!2$, 
$n=\frac{5}{12}$, in the parameter $k$-$\varepsilon$ space, 
with the red, blue and green curves (obtained from 
numerical computation) representing the saddle-node (SN), 
Hopf (H) and homoclinic loop (HL) bifurcations, respectively, 
and the blank circles indicate the parameter values for simulations, 
which are given in Figure~\ref{fig11}.} 
\label{fig10}
\end{center} 
\vspace{-0.20in} 
\end{figure}

The bifurcation diagram projected on 
a $2$-sphere is shown in Figure~\ref{fig9}.
Figure~\ref{fig9}(a) is an exact bifurcation diagram for
$\sigma=0.05$, in which the intersection points 
C, GH and DHL, as shown in Figure~\ref{fig9}(a), are given by 
$$
\begin{array}{rlll}
{\rm C} & = \ (\beta_2,\beta_3)_{\rm C}
&\!\!\!\! = (0.001880,\hspace*{0.15in}0.049947), & \ \ {\rm for} \ \
\beta_1 = -\,0.001343, \\[1.0ex]
{\rm GH} \!\! & = \, (\beta_2,\beta_3)_{\rm GH} &\!\!\!\! 
= (0.007807,\hspace*{0.15in} 0.046852),
& \ \ {\rm for} \ \ \beta_1 = -\,0.015617, \\[1.0ex]
{\rm DHL} \!\!\!\! & = (\beta_2,\beta_3)_{\rm DHL}
&\!\!\!\! = (0.003499,-\,0.049424), & \ \
{\rm for} \ \ \beta_1 = -\,0.006712. 
\end{array} 
$$

For a better view of bifurcations, a schematic 
general bifurcation diagram is shown in Fig.~\ref{fig9}(b) 
with typical phase portraits, which is similar to 
Figure 3 in \cite{Dumortier1987} and Figure 2 in \cite{LiLiMa2015}.

\begin{figure}[!t]
\vspace*{-2.70in}
\begin{center} 
\hspace*{-1.24in}
\begin{overpic}[width=1.355\textwidth,height=1.277\textheight]{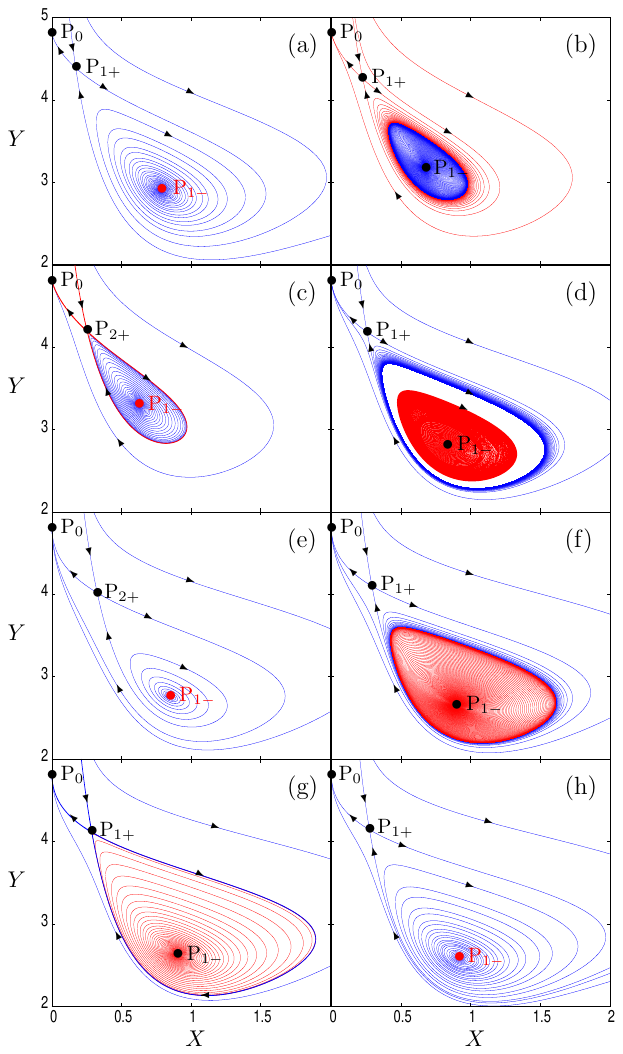}
\end{overpic}

\vspace{-1.40in}
\caption{Simulated trajectories for the epidemic model \eqref{Eqn8} with 
$m\!=\!2$, $n\!=\!\frac{5}{12}$, 
showing the bifurcations given in Figure~\ref{fig10}:
(a) the stable ${\rm P_{1-}}$ for 
$(k,\varepsilon)=(0.2439,2.2)$; (b) a stable LC 
for $(k,\varepsilon)=(0.2439,1.95)$; 
(c) the stable HL for $(k,\varepsilon)=(0.2439,1.871268)$; 
(d) 2 LC for $(k,\varepsilon)=(0.2,3.13)$;
(e) the unstable ${\rm E_{2-}}$ for $(k,\varepsilon)=(0.16202,4.2)$; 
(f) an unstable LC for $(k,\varepsilon)=(0.16202,4.44)$;
(g) the unstable HL for $(k,\varepsilon)=(0.16202,4.485125)$; and 
(h) the stable ${\rm E_{2-}}$ for $(k,\varepsilon)=(0.16202,4.6)$, 
where LC and HL represent limit cycle and homoclinic loop, respectively.}  
\label{fig11}  
\vspace*{-0.30in} 
\end{center}
\end{figure}

Finally, we present the simulation for the codimension-$3$ B-T bifurcation 
to show the dynamics described in Theorem~\ref{Thm3.3}. 
For convenience, we use the model \eqref{Eqn8} to plot the 
bifurcation diagram in the $k$-$\varepsilon$ plane. For this purpose,   
we take $m=2$, and $\mu_3 = -\frac{1}{12}$, yielding 
$$
n= n_0 + \mu_3 = \frac{1}{2}-\frac{1}{12}=\frac{5}{12}. 
\vspace{0.05in} 
$$
Then, the bifurcation diagram plotted in the $k$-$\varepsilon$ space, 
is shown in Figure~\ref{fig10}, where the red, 
blue and green curves represent the saddle-node, Hopf and homoclinic loop 
bifurcations, respectively. The green curve for the homoclinic loop  
bifurcation is obtained from numerical computation. 
It is seen that the bifurcation diagram in Figure~\ref{fig10} 
agrees well with those in Figure~\ref{fig9}, but in a reflection matter, 
namely, the stable equilibrium ${\rm E_{2-}}$ in Figure~\ref{fig10} 
appears on the right side of the Hopf bifurcation curve, while it is 
on the left side of Hopf bifurcation curve in Figure~\ref{fig9}. 
The eight blank circles in Figure~\ref{fig10} indicate the points of 
$(k,\varepsilon)$ parameter values for simulation, 
which, except for the red circle point yielding $2$ limit cycles, 
are located on the two lines: $ k = 0.16202 $ and $ k=0.2439$. 
Note that the red circle point, $(k,\varepsilon) 
= (0.2,3.13)$ for the $2$ limit cycles, is below both the blue (H) 
and green (HL) curves, since at $k=0.2$, $\varepsilon = 3.14402$ 
and $\varepsilon = 3.14519$ on the Hopf and homoclinic loop curves, 
respectively. The two green circles (on the green curves) denote the 
two homoclinic loop bifurcations, with the right one (on the 
line $k=0.2439$) stable and the left one (on the line $k=0.16202$) 
unstable. Therefore, starting from the points on 
the line $k=0.2439$ (in the downward direction) to the $2$-LC point, 
and then to the points on the line $k=0.16202$ (in the upward direction), 
we obtain the corresponding simulation figures depicted in Figure~\ref{fig11}, 
as indicated in Figure~\ref{fig10},  
which indeed, with a careful selection of the parameter values, 
demonstrates the complex bifurcation behaviours around the 
codimension-$3$ B-T bifurcation point.


\section{Conclusion}

In this paper, we have studied Hopf and Bogdanov-Takens bifurcations 
and paid particular attention to the codimension of the two bifurcations 
as well as to the dynamical behaviours around the bifurcation points. 
We have used an epidemic model to illustrate how to determine 
the codimension of Hopf and Bogdanov-Takens bifurcations. 
It has been shown that 
the difficulty mainly comes from the restriction on the system parameters. 
Moreover, for the codimension-$3$ Bogdanov-Takens bifurcation, we have 
introduced the one-step transformation approach, showing the advantage 
of this method compared to the classical six-step transformation approach.  
Numerical simulations are presented to show an excellent agreement with
the theoretical predictions.

\section*{Acknowledgement}

This research was partially supported by the 
Natural Sciences and Engineering Research Council of Canada
(NSERC No.~R2686A02).

\end{document}